\documentclass[reqno,twoside]{amsart}

\usepackage{amsfonts}
\usepackage{amsmath,amssymb}
\usepackage[dvips,bottom=1.4in,right=1in,top=1in, left=1in]{geometry}
\usepackage{epic}
\usepackage{pstricks}
\usepackage{graphicx}
\usepackage{xcolor}

\usepackage{mathrsfs,mathtools}
\usepackage{enumerate}
\usepackage{hyperref}
\usepackage{esint}
\usepackage{graphicx}
\usepackage{bm}
\usepackage{commath}
\usepackage{esint}

\usepackage{color}

\usepackage{a4wide,amsmath,amssymb,latexsym,amsthm}
\usepackage{eucal}

\usepackage{color}

\usepackage[textsize=small]{todonotes}

\newtheorem{theorem}{Theorem}[section]
\newtheorem{proposition}[theorem]{Proposition}
\newtheorem{remark}[theorem]{Remark}
\newtheorem{lemma}[theorem]{Lemma}
\newtheorem{corollary}[theorem]{Corollary} 
\newtheorem{definition}[theorem]{Definition}

\numberwithin{equation}{section}

\graphicspath{{./Imgs/}}

\newcommand{\R}{\mathbb R}
\newcommand{\C}{\mathbb C} 
\newcommand{\N}{\mathbb N}

\newcommand{\be}{\begin{equation}}
\newcommand{\ee}{\end{equation}}
\newcommand{\ba}{\begin{eqnarray}}
\newcommand{\ea}{\end{eqnarray}}
\newcommand{\beq}{\begin{equation}}
\newcommand{\eeq}{\end{equation}}

\numberwithin{equation}{section}

\def\Omc{\mathbb{R}^N\setminus\Omega}

\def\Omb{\mathbb{R}^N\setminus\overline{\Omega}}

\def\RR{{\mathbb{R}}}
\def\NN{{\mathbb{N}}}
\def\N{{\mathbb{N}}}
\def\CC{{\mathbb{C}}}
\def\C{{\mathbb{C}}}

\def\Om{\Omega}

\def\bOm{\overline{\Om}}

\def\pOm{\partial \Omega}

\usepackage[textsize=small]{todonotes}



\begin{document}

\title[Strong damping wave equations]{Analysis of the controllability from the exterior of strong damping nonlocal wave equations}\thanks{The work of the first author is partially supported by the Air Force Office of Scientific Research (AFOSR) under the Award NO:  FA9550-18-1-0242. The second author  is supported by the Fondecyt Postdoctoral Grant No. 3180322}

\author{Mahamadi Warma}
\address{M. Warma, University of Puerto Rico, Rio Piedras Campus, Department of Mathematics,
 College of Natural Sciences,  17 University AVE. STE 1701  San Juan PR 00925-2537 (USA)}
\email{mahamadi.warma1@upr.edu, mjwarma@gmail.com }

\author{Sebasti\'an Zamorano}
\address{S. Zamorano, Universidad de Santiago de Chile, Departamento de
Matem\'atica, Facultad de Ciencias, Casilla 307-Correo 2,
Santiago, Chile.}
 \email{sebastian.zamorano@usach.cl}



\begin{abstract}
We make a complete analysis of the controllability properties from the exterior of the (possible) strong damping wave equation with the fractional Laplace operator subject to the nonhomogeneous Dirichlet type exterior condition. In the first part, we show that if $0<s<1$, $\Omega\subset\RR^N$ ($N\ge 1$) is a bounded Lipschitz domain and the parameter $\delta> 0$, then there is no control function $g$ such that the  following system 
\begin{equation*}
\begin{cases}
u_{tt} + (-\Delta)^{s} u + \delta(-\Delta)^{s} u_{t}=0 & \mbox{ in }\; \Omega\times(0,T),\\
u=g\chi_{\mathcal O\times (0,T)} &\mbox{ in }\; (\Omc)\times (0,T) ,\\
u(\cdot,0) = u_0, u_t(\cdot,0) = u_1  &\mbox{ in }\; \Omega,
\end{cases}
\end{equation*}
is exact or null controllable at time $T>0$. In the second part, we prove that for every $\delta\ge 0$ and $0<s<1$, the system is indeed
approximately controllable for any $T>0$ and $g\in \mathcal D(\mathcal O\times(0,T))$, where $\mathcal O\subset\Omc$ is any non-empty open set.  
\end{abstract}

\maketitle

\section{Introduction}

Let $\Omega\subset\R^N$ ($N\ge 1$) be a bounded open set with a Lipschitz continuous boundary $\pOm$. The aim of the present paper is to study completely the controllability properties from the exterior of the (possible) strong damping nonlocal wave equation associated with the fractional Laplace operator. More precisely, we consider  the system
\begin{align}\label{SD-WE}
\begin{cases}
u_{tt} + (-\Delta)^{s} u +  \delta(-\Delta)^{s} u_{t}=0 & \mbox{ in }\; \Omega\times(0,T),\\
u=g\chi_{\mathcal O\times (0,T)} &\mbox{ in }\; (\Omc)\times (0,T), \\
u(\cdot,0) = u_0, \;\;u_t(\cdot,0) = u_1  &\mbox{ in }\; \Omega,
\end{cases}
\end{align}
where $u=u(x,t)$ is the state to be controlled, $g=g(x,t)$ is the control function which is localized on a subset $\mathcal O$ of $\Omc$, $\delta\ge 0$ and  $0<s<1$ are real numbers, and $(-\Delta)^s$ denotes the fractional Laplace operator (see \eqref{fl_def}).

We shall show that  for every $g\in L^2((0,T);W^{s,2}(\Omc))$, if $u_0$ and $u_1$ are in a suitable Banach space, then  the system \eqref{SD-WE} has a unique solution $(u,u_t)$ satisfying the regularity $u\in C([0,T];W^{s,2}(\RR^N))\cap C^1([0,T];L^2(\Om))$.
In that case,  the set of reachable states is given by
\begin{align*}
\mathcal R((u_0,u_1),T)=\Big\{(u(\cdot,T),u_t(\cdot,T)):\; g\in L^2((0,T); W^{s,2}(\Omc))\Big\}.
\end{align*}

Let $W^{-s,2}(\bOm)$ be the dual of the energy space $W_0^{s,2}(\bOm)$ (see Section \ref{preli}).

We shall consider the following three notions of controllability.

\begin{itemize}
 \item The system is said to be null controllable at $T>0$,  if 
\begin{align*}
(0,0)\in \mathcal R((u_0,u_1),T).
\end{align*}
In other words, there is a control function $g$ such that the unique solution $(u,u_t)$ satisfies $u(\cdot,T)=u_t(\cdot,T)=0$ almost everywhere in $\Omega$.

\item The system is said to be exact controllable at $T>0$, if
\begin{align*}
\mathcal R((u_0,u_1),T)=L^2(\Omega)\times W^{-s,2}(\bOm).
\end{align*}

\item The system is said to be approximately controllable at $T>0$,  if
\begin{align*}
\mathcal R((u_0,u_1),T)\;\mbox{ is dense in }\; L^2(\Omega)\times W^{-s,2}(\bOm),
\end{align*}
or equivalently, for every  $(\tilde u_0,\tilde u_1)\in L^2(\Omega)\times W^{-s,2}(\bOm)$ and $\varepsilon>0$, there is a control $g$ such that the corresponding unique solution $(u,u_t)$ of \eqref{SD-WE} with $u_0=u_1=0$ satisfies
\begin{align}\label{cond-control}
\left\|u(\cdot,T)-\tilde u_0\right\|_{L^2(\Omega)}+ \left\|u_t(\cdot,T)-\tilde u_1\right\|_{W^{-s,2}(\bOm)}\le \varepsilon.
\end{align} 
\end{itemize}
In the present article we have obtained the following specific results.

\begin{itemize}
\item[(i)] Our first main result says that if $\delta>0$, then the system is not exact or null controllable at any time $T>0$. 

\item[(ii)] We also obtain that the adjoint system associated with \eqref{SD-WE} satisfies the unique continuous property for evolution equations.

\item [(iii)] The third main result states that the system \eqref{SD-WE} is  approximately controllable, for every $\delta\ge 0$, $0<s<1$, $T>0$, and for every $g\in\mathcal D(\mathcal O\times (0,T))$, where  $\mathcal O\subset\Omc$ is any non-empty open set. Since the system is not exact or null controllable (if $\delta>0$), it is the best possible result that can be obtained regarding the controllability of such systems.
\end{itemize}

The null/exact controllabilty from the interior of the pure (without damping) wave equation (with {\em strong} zero Dirichlet exterior condition) associated with the bi-fractional Laplace operator has been investigated in \cite{Umb} by using a Pohozaev identity for the fractional Laplacian established in \cite{RS-Po}.  More precisely, the author in \cite{Umb} has considered the following problem:
\begin{align}\label{SD-WE-B}
\begin{cases}
u_{tt} + (-\Delta)^{2s} u=f \chi_{\omega\times(0,T)}& \mbox{ in }\; \Omega\times(0,T),\\
u=(-\Delta)^su=0 &\mbox{ in }\; (\Omc)\times (0,T), \\
u(\cdot,0) = u_0, \;\;u_t(\cdot,0) = u_1  &\mbox{ in }\; \Omega,
\end{cases}
\end{align}
where $0<s<1$, $u$ is the state to be controlled and $f$ is the control function localized in a certain neighborhood $\omega$ of the boundary $\pOm$.
He has shown that the system \eqref{SD-WE-B} is exact/null controllable at any time $T>0$ if $\frac 12<s<1$ and at any time $T>T_0$ if $s=\frac 12$, where $T_0$ is a certain positive constant.
We notice that, since the system \eqref{SD-WE-B} is reversible in time (which is not the case for the system \eqref{SD-WE} if $\delta>0$), then in this case, the null and exact controllabilties are the same notions.

Always in the case $\delta=0$, most recently, we have  studied in \cite{CLR-MW} the controllability of the space-time fractional wave equation, that is, the case where in  \eqref{SD-WE}, we have replaced $u_{tt}$ by the Caputo time fractional derivative $\mathbb D_t^\alpha$ $(1<\alpha<2$). 
We have obtained a positive result about the approximate controllability from the exterior. The corresponding problem for interior control has been studied in \cite{KW}.
The case of the fractional diffusion equation from the exterior, that is, when $0<\alpha\le 1$, has been completely investigated in  \cite{War-ACE}. We mention that due to the results in \cite{EZ}, fractional in time evolution equations can never be null/exact controllable.
The  null controllabilty from the interior of the heat equation associated with the fractional Laplace operator (with zero Dirichlet exterior condition) has been recently studied in one space-dimension in \cite{BiHe} by using the asymptotic gap of the eigenvalues of the realization in $L^2(\Om)$ of $(-\Delta)^s$ with the zero Dirichlet exterior condition. The case $N\ge 2$ is still an open problem. 

In the present paper, using some ideas that we have recently developed in \cite{CLR-MW,War-ACE}, we shall study the controllability of the nonlocal wave or/and the strong damping nonlocal wave equations with the control function localized at the exterior of the domain $\Omega$ where the evolution equation is solved.
 To the best of our knowledge, it is the third work (after our work \cite{War-ACE} for the case $0<\alpha\le 1$ and \cite{CLR-MW} for the case $1<\alpha<2$) that addresses the controllability of nonlocal equations from the exterior of the domain involved, and it is the first work that studies the controllability from the exterior of wave and/or strong damping nonlocal wave equations involving the fractional Laplace operator.  
 
 We also notice that from our results, when taking the limit as $s\uparrow 1^-$, we can recover the known results on the topics regarding the controllability from the boundary of the local wave or the strong damping local wave equations studied in \cite{rosier2007,Zua1} and their references. That is, the case where the control function is localized on a subset $\omega$ of  $\pOm$.

Anomalous diffusion and wave equations are of great interest in physics. In \cite{Mai} it has been shown that the fractional wave equation governs the propagation of mechanical diffusion waves in viscoelastic media. Fractional order operators have also recently emerged as a modeling alternative in various branches of science. 
A number of stochastic models for explaining anomalous diffusion have been
introduced in the literature; among them we  quote the fractional Brownian motion; the continuous time random walk;  the L\'evy flights; the Schneider grey Brownian motion; and more generally, random walk models based on evolution equations of single and distributed fractional order in  space (see e.g. \cite{DS,GR,Man,Sch,ZL}).  In general, a fractional diffusion operator corresponds to a diverging jump length variance in the random walk. We refer to \cite{NPV,Val} and the references therein for a complete analysis, the derivation and the applications of the fractional Laplace operator.
For further details we also refer to \cite{GW-F,GW-CPDE} and their references.

The rest of the paper is structured as follows. In Section \ref{sec-2} we state the main results of the article. The first one (Theorem \ref{lact-nul-cont}) says that if $\delta>0$, then the system \eqref{SD-WE} is not exact/null controllable at time $T>0$. The second main result (Theorem \ref{pro-uni-con}) shows that the adjoint system associated with \eqref{SD-WE} satisfies the unique continuation property for evolution equations and the third main result (Theorem \ref{main-Theo}) states that for every $\delta\ge 0$, the system \eqref{SD-WE} is approximately controllable at any time $T>0$. In Section \ref{preli} we introduce the function spaces needed to study our problem and we give some known results that are used in the proof of our main results. This is followed  in Section \ref{sec-4} by the proof of the existence, uniqueness, regularity and the representation of solutions of \eqref{SD-WE} and its associated dual system in terms of series. Finally in Section \ref{prof-ma-re}, we give the proof of the main results stated in Section \ref{sec-2}.

\section{Main results}\label{sec-2}

In this section we state the main results of the article. Throughout the remainder of the paper, without any mention, $\delta\ge 0$ and $0<s<1$ are real numbers and $\Omega\subset\R^N$ denotes a bounded open set with a Lipschitz continuous boundary. Given a measurable set $E\subset\RR^N$, we shall denote by $(\cdot,\cdot)_{L^2(E)}$ the scalar product in $L^2(E)$.  We refer to Section \ref{preli} for a rigorous definition of the function spaces and operators involved. Let $W_0^{s,2}(\bOm)$ be the energy space and denote by $W^{-s,2}(\bOm)$ its dual. We shall let  $\langle\cdot,\cdot\rangle_{-\frac 12,\frac 12}$ be their duality pair.

Our first main result is the following theorem.

\begin{theorem}\label{lact-nul-cont}
Let $\delta>0$. Then  the system \eqref{SD-WE} is not exact or null controllable at time $T>0$.
\end{theorem}

Next, we introduce our notion of solutions.
Let  $(u_0,u_1)\in L^2(\Omega)\times W^{-s,2}(\bOm)$ and consider the following two systems:
\begin{equation}\label{main-EQ-2}
\begin{cases}
v_{tt}+(-\Delta)^sv+\delta(-\Delta)^sv_t=0\;\;&\mbox{ in }\; \Omega\times (0,T),\\
v=g&\mbox{ in }\;(\Omc)\times (0,T),\\
v(\cdot,0)=0,\;\;v_t(\cdot,0)=0&\mbox{ in }\;\Omega,
\end{cases}
\end{equation}
and
\begin{equation}\label{main-EQ-3}
\begin{cases}
w_{tt}+(-\Delta)^sw+\delta(-\Delta)^sw_t=0\;\;&\mbox{ in }\; \Omega\times (0,T),\\
w=0&\mbox{ in }\;(\Omc)\times (0,T),\\
w(\cdot,0)=u_0,\;\;w_t(\cdot,0)=u_1&\mbox{ in }\;\Omega.
\end{cases}
\end{equation}
Then it is clear that $u=v+w$ solves the system \eqref{SD-WE}.

\begin{definition}\label{def-strong-sol}
Let  $g$ be a given function. 
A function $(v,v_t)$ is said to be a weak solution of \eqref{main-EQ-2}, if the following properties hold.

\begin{itemize}
\item Regularity:
\begin{equation}\label{regu}
\begin{cases}
v\in C([0,T];L^2(\Omega))\cap C^1([0,T];W^{-s,2}(\bOm)), \\
v_{tt}\in C((0,T); W^{-s,2}(\bOm)).
\end{cases}
\end{equation}
\item Variational identity: For every $w\in W_0^{s,2}(\bOm)$ and a.e. $t\in (0,T)$,
\begin{align*}
\langle v_{tt},w\rangle_{-\frac 12,\frac 12}+\langle (-\Delta)^s(v+\delta v_t),w\rangle_{-\frac 12,\frac 12}=0.
\end{align*}  
\item Initial and exterior conditions: 
\begin{align}\label{Var-I}
v(\cdot,0)=0, \;v_t(\cdot,0)=0\;\;\mbox{ in }\;\Omega\;\mbox{ and }\; v=g\;\mbox {in }\;(\Omc)\times (0,T).
\end{align}
\end{itemize}
\end{definition}

It follows from Definition \ref{def-strong-sol}, that for a weak solution $(v,v_t)$ of  the system \eqref{main-EQ-2}, we have that  the functions $(v(\cdot,T),v_t(\cdot,T))\in  L^2(\Omega)\times W^{-s,2}(\bOm)$.

Using the classical integration by parts formula, we have that the following backward system 
\begin{equation}\label{ACP-Dual}
\begin{cases}
\psi_{tt} +(-\Delta)^s\psi-\delta(\Delta)^s\psi_t=0\;\;&\mbox{ in }\; \Omega\times (0,T),\\
\psi=0&\mbox{ in }\;(\Omc)\times (0,T),\\
\psi(\cdot,T)=\psi_0,\;\;\;\; \psi_t(\cdot,T)=-\psi_1\;&\mbox{ in }\;\Omega,
\end{cases}
\end{equation}
is the dual system associated with \eqref{main-EQ-2}.
Our notion of weak solutions to \eqref{ACP-Dual} is as follows.

\begin{definition}
Let $(\psi_0,\psi_1) \in W_0^{s,2}(\bOm)\times L^2(\Omega)$. 
A function $(\psi,\psi_t)$ is said to be a weak  solution of \eqref{ACP-Dual}, if for a.e. $t\in (0,T)$,  the following properties hold.

\begin{itemize}
\item Regularity and final data:
\begin{equation}\label{Dual-egu}
\begin{cases}
\psi\in C([0,T];W_0^{s,2}(\bOm))\cap C^1([0,T]; L^2(\Omega)), \\
 \psi_{tt}\in C((0,T);W^{-s,2}(\bOm)),
 \end{cases}
\end{equation}
and $\psi(\cdot,T)=\psi_0$, $\psi_t(\cdot,T)=\psi_1$ in $\Omega$.
\item Variational  identity: For every $w\in W_0^{s,2}(\bOm)$ and a.e. $t\in (0,T)$, 
\begin{align*}
\langle \psi_{tt},w\rangle_{-\frac 12,\frac 12}+\langle(-\Delta)^s(\psi-\delta\psi_t),w\rangle_{-\frac 12,\frac 12}=0.
\end{align*}  
\end{itemize}
\end{definition}

The next theorem, which is our second main result, says that the adjoint system \eqref{ACP-Dual} satisfies the {\em unique continuation property for evolution equations}.

\begin{theorem}\label{pro-uni-con}
Let $(\psi_0,\psi_1)\in W_0^{s,2}(\bOm)\times L^2(\Omega)$ and $(\psi,\psi_t)$ the unique weak solution of \eqref{ACP-Dual}. Let $\mathcal O\subset\Omc$ be an arbitrary non-empty open set. If $\mathcal N_s\psi=0$ in $\mathcal O\times (0,T)$, then $\psi=0$ in $\Omega\times (0,T)$. Here, $\mathcal N_s\psi$ is the nonlocal normal derivative of $\psi$ defined in \eqref{NLND} below.
\end{theorem}

The last main result concerns the approximate controllability of \eqref{SD-WE}. For this, we notice that the study of the approximate controllability of  \eqref{SD-WE} can be reduced to the case $u_0=u_1=0$ (see e.g. \cite{KW,CLR-MW,rosier2007,War-AA,War-ACE,Zua1}). 

\begin{theorem}\label{main-Theo}
The system \eqref{SD-WE} is approximately controllable for any $T>0$ and any control function $g\in \mathcal D(\mathcal O\times(0,T))$, where $\mathcal O\subset\Omc$ is an arbitrary non-empty open set. That is,
\begin{align*}
\overline{\mathcal R((0,0),T)}^{L^2(\Omega)\times W^{-s,2}(\bOm)}=&\overline{\left\{\left(u(\cdot,T),u_t(\cdot,T)\right):\;  g\in \mathcal D(\mathcal O\times(0,T))\right\}}^{L^2(\Omega)\times W^{-s,2}(\bOm)}\\
=&L^2(\Omega)\times W^{-s,2}(\bOm),
\end{align*}
where $(u,u_t)$ is the unique weak solution of \eqref{SD-WE} with $u_0=u_1=0$.
\end{theorem}

\section{Preliminaries}\label{preli}

In this section we give some notations and recall some known results as they are needed in the proof of our main results.
We start with fractional order Sobolev spaces. Given $0<s<1$,  we let

\begin{align*}
W^{s,2}(\Omega):=\left\{u\in L^2(\Omega):\;\int_{\Omega}\int_{\Omega}\frac{|u(x)-u(y)|^2}{|x-y|^{N+2s}}\;dxdy<\infty\right\},
\end{align*}
and we endow it with the norm defined by
\begin{align*}
\|u\|_{W^{s,2}(\Omega)}:=\left(\int_{\Omega}|u(x)|^2\;dx+\int_{\Omega}\int_{\Omega}\frac{|u(x)-u(y)|^2}{|x-y|^{N+2s}}\;dxdy\right)^{\frac 12}.
\end{align*}
We set
\begin{align*}
W_0^{s,2}(\bOm):=\Big\{u\in W^{s,2}(\RR^N):\;u=0\;\mbox{ in }\;\RR^N\setminus\Omega\Big\}.
\end{align*}

For more information on fractional order Sobolev spaces, we refer to \cite{NPV,War}.

Next, we give a rigorous definition of the fractional Laplace operator. Let 
\begin{align*}
\mathcal L_s^{1}(\RR^N):=\left\{u:\RR^N\to\RR\;\mbox{ measurable},\; \int_{\RR^N}\frac{|u(x)|}{(1+|x|)^{N+2s}}\;dx<\infty\right\}.
\end{align*}
For $u\in \mathcal L_s^{1}(\RR^N)$ and $\varepsilon>0$ we set
\begin{align*}
(-\Delta)_\varepsilon^s u(x):= C_{N,s}\int_{\{y\in\RR^N:\;|x-y|>\varepsilon\}}\frac{u(x)-u(y)}{|x-y|^{N+2s}}\;dy,\;\;x\in\RR^N,
\end{align*}
where $C_{N,s}$ is a normalization constant given by
\begin{align}\label{CNs}
C_{N,s}:=\frac{s2^{2s}\Gamma\left(\frac{2s+N}{2}\right)}{\pi^{\frac
N2}\Gamma(1-s)}.
\end{align}
The {\em fractional Laplacian}  $(-\Delta)^s$ is defined by the following singular integral:
\begin{align}\label{fl_def}
(-\Delta)^su(x):=C_{N,s}\,\mbox{P.V.}\int_{\RR^N}\frac{u(x)-u(y)}{|x-y|^{N+2s}}\;dy=\lim_{\varepsilon\downarrow 0}(-\Delta)_\varepsilon^s u(x),\;\;x\in\RR^N,
\end{align}
provided that the limit exists. 
Note that $\mathcal L_s^{1}(\RR^N)$ is the right space for which $ v:=(-\Delta)_\varepsilon^s u$ exists for every $\varepsilon>0$, $v$ being also continuous at the continuity points of  $u$.  
For more details on the fractional Laplace operator we refer to \cite{BBC,Caf1,NPV,GW,GW-CPDE,GW2,War,War-In} and their references.

Next, we consider the following Dirichlet problem:
\begin{equation}\label{EDP}
\begin{cases}
(-\Delta)^su=0\;\;&\mbox{ in }\;\Omega,\\
u=g&\mbox{ in }\;\RR^N\setminus\Omega.
\end{cases}
\end{equation}

\begin{definition}\label{def-sol}
Let $g\in W^{s,2}(\Omc)$ and $G\in W^{s,2}(\RR^N)$ be such that $G|_{\Omc}=g$. A function $u\in W^{s,2}(\RR^N)$ is said to be a weak solution of \eqref{EDP} if $u-G\in W_0^{s,2}(\bOm)$ and 
\begin{align*}
\int_{\RR^N}\int_{\RR^N}\frac{(u(x)-u(y))(v(x)-v(y))}{|x-y|^{N+2s}}\;dxdy=0,\;\;\forall\; v\in W_0^{s,2}(\bOm).
\end{align*}
\end{definition}

The following existence result is taken from \cite{Grub} (see also \cite{GSU}).

\begin{proposition}\label{proposi-33}
For any $g\in W^{s,2}(\RR^N\setminus\Omega)$, there is a unique $u\in W^{s,2}(\RR^N)$ satisfying \eqref{EDP} in the sense of Definition \ref{def-sol}. In addition, there is a constant $C>0$ such that

\begin{align*} 
\|u\|_{W^{s,2}(\RR^N)}\le C\|g\|_{W^{s,2}(\RR^N\setminus\Omega)}.
\end{align*}
\end{proposition}

We consider the closed and bilinear form

\begin{align*}
\mathcal F(u,v):=\frac{C_{N,s}}{2}\int_{\RR^N}\int_{\RR^N}\frac{(u(x)-u(y))(v(x)-v(y))}{|x-y|^{N+2s}}\;dxdy,\;\;u,v\in W_0^{s,2}(\bOm).
\end{align*}
Let $(-\Delta)_D^s$ be the selfadjoint operator in $L^2(\Omega)$ associated with $\mathcal F$ in the sense that

\begin{equation*}
\begin{cases}
D((-\Delta)_D^s)=\Big\{u\in W_0^{s,2}(\bOm),\;\exists\;f\in L^2(\Omega),\;\mathcal F(u,v)=(f,v)_{L^2(\Omega)}\;\forall\;v\in W_0^{s,2}(\bOm)\Big\},\\
(-\Delta)_D^su=f.
\end{cases}
\end{equation*}
More precisely, we have that

\begin{equation*} 
D((-\Delta)_D^s):=\left\{u\in W_0^{s,2}(\bOm),\; (-\Delta)^su\in L^2(\Omega)\right\},\;\;\;
(-\Delta)_D^su:=(-\Delta)^su.
\end{equation*}
Then $(-\Delta)_D^s$ is the realization of $(-\Delta)^s$ in $L^2(\Omega)$ with the condition $u=0$ in $\Omc$. It is well-known (see e.g. \cite{SV2,GW-F,War-ACE}) that $(-\Delta)_D^s$ has a compact resolvent and its eigenvalues form a non-decreasing sequence of real numbers $0<\lambda_1\le\lambda_2\le\cdots\le\lambda_n\le\cdots$ satisfying $\lim_{n\to\infty}\lambda_n=\infty$.  In addition, the eigenvalues are of finite multiplicity.
Let $(\varphi_n)_{n\in\NN}$ be the orthonormal basis of eigenfunctions associated with $(\lambda_n)_{n\in\NN}$. Then $\varphi_n\in D((-\Delta)_D^s)$ for every $n\in\NN$,  $(\varphi_n)_{n\in\NN}$ is total in $L^2(\Omega)$ and satisfies 
\begin{equation}\label{ei-val-pro}
\begin{cases}
(-\Delta)^s\varphi_n=\lambda_n\varphi_n\;\;&\mbox{ in }\;\Omega,\\
\varphi_n=0\;&\mbox{ in }\;\Omc.
\end{cases}
\end{equation}
With this setting, we have that
for $u\in W_0^{s,2}(\bOm)$, 
\begin{align}\label{norm-V}
\|u\|_{W_0^{s,2}(\bOm)}^2:=\sum_{n=1}^\infty\left|\lambda_n^{\frac 12}\left(u,\varphi_n\right)_{L^2(\Om)}\right|^2,
\end{align}
defines an equivalent norm on $W_0^{s,2}(\bOm)$. If $u\in D((-\Delta)_D^s)$, then
\begin{align*}
\|u\|_{D((-\Delta)_D^s)}^2=\|(-\Delta)_D^su\|_{L^2(\Omega)}^2=\sum_{n=1}^\infty\left|\lambda_n\left(u,\varphi_n\right)_{L^2(\Om)}\right|^2.
\end{align*}
In addition, for $u\in W^{-s,2}(\bOm)$, we have that
\begin{align}\label{norm-V-2}
\|u\|_{W^{-s,2}(\bOm)}^2=\sum_{n=1}^\infty\left|\lambda_n^{-\frac 12}\left(u,\varphi_n\right)_{L^2(\Om)}\right|^2.
\end{align}  
In that case, by the Gelfand triple (see e.g. \cite{ATW}), we have $W_0^{s,2}(\bOm)\hookrightarrow L^2(\Omega)\hookrightarrow W^{-s,2}(\bOm)$. 

Next, for $u\in W^{s,2}(\RR^N)$ we introduce the {\em nonlocal normal derivative $\mathcal N_s$} given by 
\begin{align}\label{NLND}
\mathcal N_su(x):=C_{N,s}\int_{\Omega}\frac{u(x)-u(y)}{|x-y|^{N+2s}}\;dy,\;\;\;x\in\RR^N\setminus\bOm,
\end{align}
where $C_{N,s}$ is the constant given in \eqref{CNs}.
By \cite[Lemma 3.2]{GSU},  for every $u\in W^{s,2}(\RR^N)$, we have that $\mathcal N_su\in L^2(\Omc)$.

The following unique continuation property which shall play an important role in the proof of our main result has been recently obtained in \cite[Theorem 3.10]{War-ACE}.

\begin{lemma}\label{lem-UCD}
Let $\lambda>0$ be a real number  and $\mathcal O\subset\Omb$ a non-empty open set. 
If $\varphi\in D((-\Delta)_D^s)$ satisfies
\begin{equation*}
(-\Delta)_D^s\varphi=\lambda\varphi\;\mbox{ in }\;\Omega\;\mbox{ and }\; \mathcal N_s\varphi=0\;\mbox{ in }\;\mathcal O,
\end{equation*}
then $\varphi=0$ in $\RR^N$. 
\end{lemma}

\begin{remark}
{\em The following important identity has been recently proved in \cite[Remark 3.11]{War-ACE}.
Let $g\in W^{s,2}(\RR^N\setminus\Omega)$ and $U_g\in W^{s,2}(\RR^N)$ the associated unique weak solution of  \eqref{EDP}. Then
\begin{align}\label{eqA9}
\int_{\Omc}g\mathcal N_s\varphi_n\;dx=-\lambda_n\int_{\Omega}\varphi_nU_g\;dx.
\end{align}
}
\end{remark}

For more details on the Dirichlet problem associated with the fractional Laplace operator we refer the interested reader to \cite{BWZ2,BWZ1,BBC,Caf1,Grub,RS2-2,RS-DP,War,War-ACE} and their references.

The following integration by parts formula is contained in \cite[Lemma 3.3]{DRV} for smooth functions. The version given here can be obtained by using a simple density argument (see e.g. \cite{War-ACE}).

\begin{proposition}
 Let $u\in W^{s,2}(\RR^N)$ be such that $(-\Delta)^su\in L^2(\Omega)$. Then for every $v\in W^{s,2}(\RR^N)$, we have
\begin{align}\label{Int-Part}
\frac{C_{N,s}}{2}\int\int_{\RR^{2N}\setminus(\Omc)^2}&\frac{(u(x)-u(y))(v(x)-v(y))}{|x-y|^{N+2s}}\;dxdy\notag\\
=&\int_{\Omega}v(-\Delta)^su\;dx+\int_{\Omc}v\mathcal N_su\;dx.
\end{align}
\end{proposition}

\begin{remark}
{\em We mention the following facts.
\begin{enumerate}
\item Firstly, we notice that 
\begin{align*}
\RR^{2N}\setminus(\RR^N\setminus\Omega)^2=(\Omega\times\Omega)\cup (\Omega\times(\RR^N\setminus\Omega))\cup((\RR^N\setminus\Omega)\times\Omega).
\end{align*}

\item Secondly, if $u=0$ in $\Omc$ or $v=0$ in $\Omc$, then
\begin{align*}
\int\int_{\RR^{2N}\setminus(\RR^N\setminus\Omega)^2}\frac{(u(x)-u(y))(v(x)-v(y))}{|x-y|^{N+2s}}dxdy=\int_{\R^N}\int_{\R^N}\frac{(u(x)-u(y))(v(x)-v(y))}{|x-y|^{N+2s}}dxdy,
\end{align*}
so that the identity \eqref{Int-Part} becomes
\begin{align}\label{Int-Part-2}
\frac{C_{N,s}}{2}\int_{\R^N}\int_{\R^N}\frac{(u(x)-u(y))(v(x)-v(y))}{|x-y|^{N+2s}}\;dxdy
=\int_{\Omega}v(-\Delta)^su\;dx+\int_{\Omc}v\mathcal N_su\;dx.
\end{align}
\end{enumerate}
}
\end{remark}

We conclude this section with the following convergence result.

\begin{lemma}\label{lem-37}
Let $u\in W_0^{1,2}(\Omega)\hookrightarrow W_0^{s,2}(\bOm)$ be such that $(-\Delta)^su, \Delta u\in L^2(\Omega)$. Then the following assertions hold.
\begin{enumerate}
\item For every $v\in W_0^{1,2}(\Omega)$, 
\begin{align}\label{lim1}
\lim_{s\uparrow 1^-}\int_{\Omega}v(-\Delta)^su\;dx=-\int_{\Omega}v\Delta u\;dx.
\end{align}
\item For every $v\in W^{1,2}(\RR^N)$, 

\begin{align}\label{lim2}
\lim_{s\uparrow 1^-}\int_{\Omc}v\mathcal N_su\;dx=\int_{\pOm}v\partial_{\nu}u\;d\sigma,
\end{align}
where $\partial_{\nu}u$ is the normal derivative of $u$ in direction of the outer normal vector $\vec\nu$.
\end{enumerate}
\end{lemma}

We refer to \cite[Proposition 2.2]{BPS} for Part (a) and \cite[Proposition 5.1]{DRV} for Part (b).

\section{Series representation of solutions}\label{sec-4}

In this section we give a representation in terms of series of weak solutions to the system \eqref{main-EQ-2} and the dual system \eqref{ACP-Dual}. Evolution equations with non-homogeneous boundary or exterior conditions are in general not so easy to solve since one cannot apply directly semigroup methods due the fact that the associated operator is in general not a generator of a semigroup. For this reason, we shall give more details in the proofs.
The representation of solutions in terms of series shall play an important role in the proof of our main results. 

We recall that $(\varphi_n)_{n\in\NN}$ denotes the orthornormal basis of eigenfunctions of the operator $(-\Delta)_D^s$ associated with the eigenvalues $(\lambda_n)_{n\in\NN}$. 

Let $\delta\ge 0$ and set
\begin{align}
{\bf D_n^\delta}:=\delta^2\lambda_n^2-4\lambda_n.
\end{align}
We have the following two situations.
\begin{itemize}
\item If $\delta>0$, since $0<\lambda_1\le\lambda_2\le\cdots\le\lambda_n\le\cdots$ and $\lim_{n\to\infty}\lambda_n=+\infty$, it follows that there is a number $N_0\in\NN_0:=\NN\cup\{0\}$ such that $\delta^2\lambda_n<4$ for all $n\le N_0$.
In that case we shall use the following notations.
\begin{enumerate}
\item If ${\bf D_n^\delta}\ge 0$, that is, if $\delta^2\lambda_n-4\ge 0$, then we let
\begin{align}\label{lam}
\lambda_n^{\pm}:=\frac{-\delta\lambda_{n}\pm\sqrt{\bf D_n^\delta}}{2}.
\end{align}

\item if ${\bf D_n^\delta}<0$, that is, if $\delta^2\lambda_n-4<0$, then we let 
\begin{align}\label{lam1}
\widetilde{\lambda}_n^{\pm}:=\frac{-\delta\lambda_{n}\pm i\sqrt{-\bf D_n^\delta}}{2},\quad \alpha_n:=\mbox{Re}(\widetilde{\lambda}_n^+)=\frac{-\delta\lambda_{n}}{2}\;\mbox{ and }\;\beta_n=\mbox{Im}(\widetilde{\lambda}_n^+)=\frac{\sqrt{-\bf D_n^\delta}}{2}.
\end{align}
\end{enumerate}
\item If $\delta=0$, then ${\bf D_n^0}:=-4\lambda_n<0$ for all $n\in\NN$. In that case we let
\begin{align}\label{lam2}
\widetilde{\lambda}_n^{\pm}:=\pm i\sqrt{\lambda_n},\;\; \alpha_n=0\;\mbox{ and }\;\beta_n=\sqrt{\lambda_n}.
\end{align}
\end{itemize}

\begin{remark}\label{rem-41}
{\em An immediate and important consequence is the following. If ${\bf D_n^\delta}\ge 0$, then we have that $\lambda_n^{\pm}< 0$ for all $n>N_0$, and 
\begin{align}\label{convergence}
\lambda_n^+\to -\delta, \quad \lambda_n^-\to-\infty,\,\;\text{ as }\,\, n\to\infty.
\end{align}
This fact will be used in the proof of our main results.} 
\end{remark}

\subsection{Series solutions of the system \eqref{main-EQ-2}}

 Recall that we have shown in Section \ref{sec-2} that a solution $(u,u_t)$ of \eqref{SD-WE} can be written as $u=v+w$ where $(v,v_t)$ solves \eqref{main-EQ-2} and $(w,w_t)$ is a solution of \eqref{main-EQ-3}. Let $\delta\ge 0$ and consider the system \eqref{main-EQ-3}. That is,
 
\begin{equation}\label{main-EQ-3-2}
\begin{cases}
w_{tt}+(-\Delta)^sw+\delta(-\Delta)^sw_t=0\;\;&\mbox{ in }\; \Omega\times (0,T),\\
w=0&\mbox{ in }\;(\Omc)\times (0,T),\\
w(\cdot,0)=u_0,\;\;w_t(\cdot,0)=u_1&\mbox{ in }\;\Omega.
\end{cases}
\end{equation} 
Let 
\begin{equation*}
W=\left(
\begin{array}{c}
w\\w_t \\ 
\end{array}
\right)\;\;\mbox{ and }\;   W_0=\left(
\begin{array}{c}
u_0\\u_1 \\ 
\end{array}
\right).
\end{equation*}
Then \eqref{main-EQ-3-2} can be rewritten as the following first order Cauchy problem:

\begin{equation}\label{equa27}
\begin{cases}
W_{t}+\mathcal A_\delta W=0\;\;&\mbox{ in }\; \Omega\times (0,T),\\
W(\cdot,0)=W_0&\mbox{ in }\;\Omega,
\end{cases}
\end{equation}
where the operator matrix $\mathcal A_\delta$ with domain $D(\mathcal A_\delta)=D((-\Delta)_D^s)\times D((-\Delta)_D^s)$ is given by 

\begin{equation}\label{eq-42}
\mathcal A_\delta:=\left(
\begin{array}{cc}
0&-I\\
(-\Delta)_D^s&\delta(-\Delta)_D^s
\end{array}
\right).
\end{equation}
Let $\mathcal H:=W_0^{s,2}(\bOm)\times L^2(\Omega)$ be the Hilbert space equipped with the scalar product

\begin{align*}
\langle (v_1,v_2),(w_1,w_2)\rangle_{\mathcal H}:=&\int_{\Om}v_1w_1\;dx+\int_{\Omega}v_2w_2\;dx\\
&+\int_{\RR^N}\int_{\R^N}\frac{(v_1(x)-v_1(y))(w_1(x)-w_1(y))}{|x-y|^{N+2s}}\;dxdy.
\end{align*}

The following result is classical in an abstract form. We include the proof for the sake of completeness. 

\begin{lemma}\label{lem-SG}
The operator $-\mathcal A_\delta$ generates a strongly continuous semigroup on $\mathcal H$.
\end{lemma}

\begin{proof}
We prove the lemma in several steps. We shall apply the Lumer-Phillips theorem to the operator $\mathcal B_\delta=-\mathcal A_\delta-I$.\\

{\bf Step 1}: We claim that $\mathcal B_\delta$ is a closed operator. Indeed, assume that $U_n\in D(\mathcal B_\delta)$ satisfies $U_n\to U$ in $\mathcal H$ and $\mathcal B_\delta U_n\to V$ in $\mathcal H$, as $n\to\infty$, where
\begin{equation*}
U_n:=\left(
\begin{array}{c}
u_1^n\\u_2^n \\ 
\end{array}
\right),\;\;\;\;   U:=\left(
\begin{array}{c}
u_1\\u_2 \\ 
\end{array}
\right)\;\mbox{ and }\; V:=\left(
\begin{array}{c}
v_1\\v_2 \\ 
\end{array}
\right).
\end{equation*}
This means that $u_1^n\to u_1$ in $W_0^{s,2}(\bOm)$, $u_2^n\to u_2$ in $L^2(\Omega)$ and
\begin{equation*}
\begin{cases}
-u_1^n+u_2^n\to v_1\;\;&\mbox{ in }\; W_0^{s,2}(\bOm),\\
-u_2^n-(-\Delta)_D^s(u_1^n+\delta u_2^n)\to v_2\;&\mbox{ in }\; L^2(\Omega),
\end{cases}
\end{equation*}
as $n\to\infty$. Thus $u_1^n+u_2^n\to u_1+u_2$ and $(-\Delta)_D^s(u_1^n+\delta u_2^n)\to -(u_2+v_2)$, in $L^2(\Omega)$, as $n\to\infty$. Recall that $(-\Delta)_D^s$ is a closed operator in $L^2(\Omega)$. So, $u_1+\delta u_2\in D((-\Delta)_D^s)$ and $(-\Delta)_D^s(u_1+\delta u_2)=-(u_2+v_2)$. We have shown that $U\in D(\mathcal B_\delta)=D(\mathcal A_\delta)$ and $\mathcal B_\delta(U)=V$. Hence, the operator $\mathcal B_\delta$ is closed.\\

{\bf Step 2}: We show that $\mathcal B_\delta$ is a dissipative operator, that is, 
\begin{align}\label{dissip}
\lambda\|U\|_{\mathcal H}\le \|\mathcal B_\delta U-\lambda U\|_{\mathcal H}\;\mbox{ for any }\; \lambda>0\;\mbox{ and }\; U\in D(\mathcal B_\delta).
\end{align}
Let $\lambda>0$ and $U\in D(\mathcal B_\delta)$. Then
\begin{align}\label{e-L}
\langle \mathcal B_\delta U, U\rangle_{\mathcal H}=&-(u_1,u_1)_{L^2(\Omega)}+(u_2,u_1)_{L^2(\Omega)}\notag\\
&-((-\Delta)_D^{s/2}u_1,(-\Delta)_D^{s/2}u_1)_{L^2(\Omega)}+ ((-\Delta)_D^{s/2}u_2,(-\Delta)_D^{s/2}u_1)_{L^2(\Omega)}\notag\\
&-((-\Delta)_D^{s}u_1,u_2)_{L^2(\Omega)}-\delta((-\Delta)_D^{s}u_2,u_2)_{L^2(\Omega)}-(u_2,u_2)_{L^2(\Omega)}.
\end{align}
Since $|(u,v)_{L^2(\Omega)}|\le \frac 12\left(\|u\|_{L^2(\Omega)}^2+\|v\|_{L^2(\Omega)}^2\right)$, it follows from \eqref{e-L} that $\langle \mathcal B_\delta U, U\rangle_{\mathcal H}\le 0$. Thus
\begin{align*}
\langle\mathcal B_\delta U-\lambda U,U\rangle_{\mathcal H}=\langle\mathcal B_\delta U,U\rangle_{\mathcal H}-\lambda\|U\|_{\mathcal H}^2\le -\lambda\|U\|_{\mathcal H}^2,\;\;U\in D(\mathcal B_\delta).
\end{align*}
This implies that $\lambda\|U\|_{\mathcal H}^2\le |\langle\mathcal B_\delta U-\lambda U,U\rangle_{\mathcal H}|\le \|\mathcal B_\delta U-\lambda U\|_{\mathcal H}\|U\|_{\mathcal H}$ and we have shown \eqref{dissip}.\\

{\bf Step 3}: Since $D((-\Delta)_D^s)$ is dense in $L^2(\Omega)$ and in $W_0^{s,2}(\bOm)$, then $D(\mathcal B_\delta)$ is dense in $\mathcal H$.\\

{\bf Step 4}: We show that $\mathcal R(I-\mathcal B_\delta)=\mathcal H$. We have to solve the equation $(I-\mathcal B_\delta)U=V$ for every $V\in\mathcal H$.	 That is,
\begin{align*}
2u_1-u_2&=v_1\\
2u_2+(-\Delta)_D^s(u_1+\delta u_2)&=v_2.
\end{align*}
Since $(-\Delta)_D^s$ is a non-negative selfadjoint operator in $L^2(\Omega)$, a simple calculation gives
\begin{align*}
u_1=&\Big((-\Delta)_D^s+2\delta I\Big)\Big(4I+(1+2\delta)(-\Delta)_D^s\Big)^{-1}v_1+\Big(4I+(1+2\delta)(-\Delta)_D^s\Big)^{-1}v_2\\
u_2=&\Big(4(\delta-1)I+(1-2\delta)(-\Delta)_D^s\Big)\Big(4I+3(-\Delta)_D^s\Big)^{-1}v_1+2\Big(4I+(1+2\delta)(-\Delta)_D^s\Big)^{-1}v_2,
\end{align*}
and we have proved that  $\mathcal R(I-\mathcal B_\delta)=\mathcal H$.

Finally, it follows from the Lumer-Phillips theorem (see e.g. \cite[Theorem 3.4.5]{ABHN}), that the operator $\mathcal B_\delta$ generates a strongly continuous semigroup $(e^{t\mathcal B_\delta})_{t\ge 0}$ on $\mathcal H$ and thus, the operator $-\mathcal A_\delta$ generates the strongly continuous semigroup $(e^{-t\mathcal A_\delta})_{t\ge 0}=(e^te^{t\mathcal B_\delta})_{t\ge 0}$ on $\mathcal H$. The proof is finished.
\end{proof}

\begin{remark}
{\em 
As a consequence of Lemma \ref{lem-SG}, it follows from semigroup theory (see e.g. \cite[Chapter 3]{ABHN}) that, for every $(u_0,u_1)\in W_0^{s,2}(\bOm)\times L^2(\Omega)$, the first order Cauchy problem \eqref{equa27} has a unique strong solution $W\in C([0,T];\mathcal H)$. Hence, the system \eqref{main-EQ-3-2}, has a unique (mild) solution $(w,w_t)$ satisfying
\begin{align}\label{regulari}
w\in C([0,T];W_0^{s,2}(\bOm))\cap C^1([0,T];L^2(\Omega)).
\end{align}
}
\end{remark}

Next we give the representation of solutions in terms of series.

\begin{proposition}\label{pro-sol-SG}
Let  $(u_0, u_1)\in W_0^{s,2}(\bOm)\times L^2(\Omega)$. Then the solution $(w,w_t)$ of \eqref{main-EQ-3-2} is given by
 \begin{align}\label{eq-9}
w(x,t)=\sum_{n=1}^{\infty}\Big(A_{n}(t)(u_{0},\varphi_{n})_{L^2(\Om)}+B_{n}(t)(u_{1},\varphi_n)_{L^2(\Om)}\Big)\varphi_{n}(x),
\end{align}
where
\begin{align}\label{10}
A_{n}(t):=
\begin{cases}
\displaystyle \Big(\cos(\beta_n t)-\frac{\alpha_n}{\beta_n}\sin(\beta_n t)\Big)e^{\alpha_n t} & \mbox{ if }  n\leq N_0,\\
\displaystyle\frac{\lambda_n^-e^{\lambda_n^+ t}-\lambda_n^+e^{\lambda_n^-t}}{\lambda_n^--\lambda_n^+} & \mbox{ if }  n>N_0,
\end{cases}
\end{align}
and
\begin{align}\label{11}
B_{n}(t):=
\begin{cases}
\displaystyle\frac{\sin(\beta_n t)}{\beta_n}e^{\alpha_n t} & \mbox{ if }  n\leq N_0,\\
\displaystyle\frac{e^{\lambda_n^- t}-e^{\lambda_n^+t}}{\lambda_n^--\lambda_n^+} & \mbox{ if }  n>N_0.
\end{cases}
\end{align}
Here, $\lambda_n^{\pm}$, $\alpha_n$ and $\beta_n$ are the real numbers given in \eqref{lam}, \eqref{lam1} and \eqref{lam2}.
\end{proposition}

\begin{proof}
 Using the spectral theorem of selfadjoint operators, we can proceed as follows. We look for a solution $(w,w_t)$ of \eqref{main-EQ-3-2} in the form

\begin{align}\label{3}
w(x,t)=\sum_{n=1}^{\infty}\left( w(\cdot,t),\varphi_{n}\right)_{L^2(\Omega)} \varphi_{n}(x).
\end{align}
For the sake of simplicity we let $w_{n}(t)=\left( w(\cdot,t),\varphi_{n}\right)_{L^2(\Omega)}$. Replacing \eqref{3} in the first equation of \eqref{main-EQ-3-2}, then multiplying both sides with $\varphi_k$ and integrating over $\Omega$, we get that $w_n(t)$ satisfies the following ordinary differential equation:
\begin{align}\label{4}
w_{n}''(t)+\lambda_{n}w_{n}(t)+\delta\lambda_{n}w_{n}'(t)=0.
\end{align}
Solving \eqref{4}, calculating, letting $u_{0,n}=(u_0,\varphi_n)_{L^2(\Om)}$ and  $u_{1,n}=(u_1,\varphi_n)_{L^2(\Om)}$, we get 

\begin{align}\label{5}
w(x,t)=\sum_{n=1}^{N_0}\Big(a_{n}^1\cos(\beta_n t)+a_{n}^2\sin(\beta_n t)\Big)e^{\alpha_n t}\varphi_{n}(x)+\sum_{n=N_0+1}^{\infty}\Big(a_{n}^3e^{\lambda_n^+t}+a_{n}^4e^{\lambda_n^-t}\Big)\varphi_{n}(x),
\end{align}
where
\begin{align*}  
a_n^1=u_{0,n}, \quad a_n^2=\frac{u_{1,n}-\alpha_n u_{0,n}}{\beta_n}\;\mbox{ for }\; n\le N_0,
\end{align*}
and
\begin{align*} 
a_n^3=\frac{u_{0,n}\lambda_n^--u_{1,n}}{\lambda_n^--\lambda_n^+},\quad a_n^4=\frac{u_{1,n}-u_{0,n}\lambda_n^+}{\lambda_n^--\lambda_n^+} \;\mbox{ for }\; n>N_0.
\end{align*}
Therefore, we obtain the following expression of $w$:
\begin{align}\label{9}
w(x,t)=&\sum_{n=1}^{N_0}\left[u_{0,n}\left(\cos(\beta_n t)-\frac{\alpha}{\beta_n}\sin(\beta_n t)\right)+u_{1,n}\frac{\sin(\beta_n t)}{\beta}\right]e^{\alpha_n t}\varphi_{n}(x)\notag\\
&+\sum_{n=N_0+1}^{\infty}\left[u_{0,n}\left(\frac{\lambda_n^-e^{\lambda_n^+ t}-\lambda_n^+e^{\lambda_n^-t}}{\lambda_n^--\lambda_b^+}\right)+u_{1,n}\left(\frac{e^{\lambda_n^- t}-e^{\lambda_n^+t}}{\lambda_n^--\lambda_n^+}\right)\right]\varphi_{n}(x).
\end{align}
Let $A_n(t)$ and $B_n(t)$ be given in \eqref{10} and \eqref{11}, respectively. Then  \eqref{eq-9} follows from  \eqref{9}.  

A simple calculation gives
\begin{align*}
w(x,0)=&\sum_{n=1}^{\infty}\Big(A_{n}(0)u_{0,n}+B_{n}(0)u_{1,n}\Big)\varphi_{n}(x)=\sum_{n=1}^{\infty}u_{0,n}\varphi_{n}(x)=u_0(x),
\end{align*}
and 
\begin{align*}
w_t(x,0)=\sum_{n=1}^{\infty}\Big(A_{n}'(0)u_{0,n}+B_{n}'(0)u_{1,n}\Big)\varphi_{n}(x)=\sum_{n=1}^{\infty}u_{1,n}\varphi_{n}(x)=u_1(x).
\end{align*}
It is straightforward to verify that $w$ given in \eqref{eq-9} has the regularity \eqref{regulari}. Since we are not interested with solutions of \eqref{main-EQ-3-2}, we will not go into details. The proof is finished.
\end{proof}

Next, we consider the non-homogeneous system \eqref{main-EQ-2}, that is, 

\begin{align}\label{main-EQ-2-2}
\begin{cases}
v_{tt} + (-\Delta)^{s} v + \delta(-\Delta)^{s} v_{t}=0 &\mbox{ in }\; \Omega\times(0,T),\\
v=g & \mbox{ in }\; (\Omc)\times (0,T), \\
v(\cdot,0) = 0, v_t(\cdot,0) = 0  &\mbox{ in }\; \Omega.
\end{cases}
\end{align}
We have the following result.

\begin{theorem}\label{theo-44}
For every $g\in\mathcal D((\Omc)\times (0,T))$, the system \eqref{main-EQ-2-2} has a unique weak (classical solution) $(v,v_t)$ such that $v\in C^\infty([0,T];W^{s,2}(\RR^N))$ and is given by
\begin{align}\label{eq-22-1}
v(x,t)=\sum_{n=1}^{\infty}\left(\int_0^t \Big( g(\cdot,\tau),\mathcal{N}_{s}\varphi_{n} \Big)_{L^2(\Omc)}\frac{1}{\lambda_{n}}B_{n}''(t-\tau)d\tau\right)\varphi_{n}(x).
\end{align}
Moreover, there is a constant $C>0$ such that for all $t\in [0,T]$ and $m\in\NN_0$, 
\begin{align}\label{ESt-Ind}
\|\partial_t^mv(\cdot,t)\|_{W^{s,2}(\RR^N)}\le C\left(\|\partial_t^{m+2}g\|_{L^\infty((0,T);W^{s,2}(\Omc))}+\|\partial_t^mg(\cdot,t)\|_{W^{s,2}(\Omc)}\right).
\end{align}
\end{theorem}

\begin{proof}
We proof the theorem in several steps.\\

{\bf Step 1}: Consider the following elliptic Dirichlet exterior problem:

\begin{align}\label{13}
\begin{cases}
 (-\Delta)^{s} \phi =0 & \Omega,\\
\phi=g & \ \Omc.
\end{cases}
\end{align}
We have shown in Proposition \ref{proposi-33} that for every $g\in W^{s,2}(\Omc)$,
there exists a unique function $\phi\in W^{s,2}(\RR^N)$ solution of \eqref{13}, and there is a constant $C>0$ such that
\begin{align}\label{E-DP-D}
\|\phi\|_{W^{s,2}(\RR^N)}\le C\|g\|_{W^{s,2}(\Omc)}.
\end{align}
Since $g$ depends on $(x,t)$, then $\phi$ also depends on $(x,t)$.  If in \eqref{13} one replaces $g$ by $\partial_t^mg$, $m\in\NN$, then the associated unique solution is given by $\partial_t^m\phi$ for every $m\in\NN_0$. From this, we can deduce that $\phi\in C^\infty([0,T]; W^{s,2}(\RR^N))$.

Now let $(v,v_t)$ be a solution of \eqref{main-EQ-2-2} and set $w:=v-\phi$. Then a simple calculation gives

\begin{align*}
w_{tt}+(-\Delta)^sw+\delta(-\Delta)^sw_t=&v_{tt}-\phi_{tt}+(-\Delta)^sv-(-\Delta)^s\phi+\delta(-\Delta)^sv_t-\delta(-\Delta)^s\phi_t\\
=&v_{tt}+(-\Delta)^sv+\delta(-\Delta)^sv_t-\phi_{tt}=-\phi_{tt} \;\;\mbox{ in }\;\Omega\times(0,T).
\end{align*}
In addition

\begin{align*}
w=v-\phi=g-g=0\;\mbox{ in }\;(\Omc)\times (0,T),
\end{align*}
and 

\begin{equation*}
\begin{cases}
w(\cdot,0)=v(\cdot,0)-\phi(\cdot,0)=-\phi(\cdot,0)\;\;\;&\mbox{ in }\;\Omega,\\
w_t(\cdot,0)=v_t(\cdot,0)-\phi_t(\cdot,0)=-\phi_t(\cdot,0)\;\;\;\;&\mbox{ in }\;\Omega.
\end{cases}
\end{equation*}
Since $g\in \mathcal D((\Omc)\times (0,T))$, we have that $\phi(\cdot,0)=\partial_t\phi(\cdot,0)=0$ in $\Omega$.  We have shown that a solution $(v,v_t)$ of \eqref{main-EQ-2-2} can be decomposed as $v=\phi+w$, where  $(w,w_t)$ solves the system
\begin{equation}\label{eq-w-D}
\begin{cases}
w_{tt}+(-\Delta)^sw+\delta(-\Delta)^sw_t=-\phi_{tt}\;\;&\mbox{ in }\;\Omega\times (0,T),\\
w=0  &\mbox{ in }\;(\Omc)\times (0,T),\\
w(\cdot,0)=0,\;\partial_tw(\cdot,0)=0\;\;&\mbox{ in }\;\Omega.
\end{cases}
\end{equation}
We notice that $\phi_{tt}\in C^\infty([0,T];W^{s,2}(\RR^N))$. \\

{\bf Step 2}: We observe that letting 

\begin{equation*}
W=\left(
\begin{array}{c}
w\\w_t \\ 
\end{array}
\right)\;\;\mbox{ and }\;   \Phi_{tt}=\left(
\begin{array}{c}
0\\-\phi_{tt}\\ 
\end{array}
\right),
\end{equation*}
then the system \eqref{eq-w-D} can be rewritten as the following first order Cauchy problem
\begin{equation}\label{eq-CP}
\begin{cases}
W_t+\mathcal A_\delta W=\Phi_{tt}\;\;&\mbox{ in }\;\Omega\times(0,T),\\
W(0)=\left(\begin{array}{c}
0\\0\\ 
\end{array}
\right)&\mbox{ in }\;\Omega,
\end{cases}
\end{equation}
where $\mathcal A_\delta$ is the matrix operator defined in \eqref{eq-42}. Proceeding as the proof of Proposition \ref{pro-sol-SG} and using semigroup theory, we get that \eqref{eq-CP} has a unique classical solution $W$ and hence,
\eqref{eq-w-D} has a unique weak (classical) solution $(w,w_t)$ such that $w\in C^\infty([0,T];W^{s,2}(\RR^N))$ and is given by

\begin{align}\label{eq-17}
w(x,t)=-\sum_{n=1}^{\infty}\left(\int_0^t \left( \phi_{\tau\tau}(\cdot,\tau),\varphi_{n}\right)_{L^2(\Omega)}B_{n}(t-\tau)d\tau\right)\varphi_{n}(x),
\end{align}
where we recall that $B_n$ is given in \eqref{11}.
Integrating  \eqref{eq-17} by parts we get that
\begin{align*}
w(x,t)=&-\sum_{n=1}^{\infty}\left(\int_0^t ( \phi(\cdot,\tau),\varphi_{n})_{L^2(\Omega)}B_{n}''(t-\tau)d\tau\right)\varphi_{n}(x)\notag\\
&-\sum_{n=1}^{\infty}\left((\phi_{\tau}(\cdot,\tau),\varphi_{n})_{L^2(\Omega)}B_{n}(t-\tau)\Big|_0^t\right)\varphi_{n}(x)\notag\\
&-\sum_{n=1}^{\infty}\left((\phi(\cdot,\tau),\varphi_{n})_{L^2(\Omega)}B_{n}'(t-\tau)\Big|_0^t\right)\varphi_{n}(x).
\end{align*}
We observe that
$B_{n}(0)=0$ and $B_{n}'(0)=1$ for all $n\in\N$. Since $\phi(\cdot,0)=\phi_{t}(\cdot,0)=0$, we get 
\begin{align}\label{eq-20}
w(x,t)=-\phi(x,t)-\sum_{n=1}^{\infty}\left(\int_0^t ( \phi(\cdot,\tau),\varphi_{n})_{L^2(\Omega)}B_{n}''(t-\tau)d\tau\right)\varphi_{n}(x).
\end{align}
Using the fact that $\varphi_{n}$ satisfies \eqref{ei-val-pro} and the integration by parts formula \eqref{Int-Part}-\eqref{Int-Part-2}, we get 

\begin{align}\label{eq-21}
\Big(\phi(\cdot,\tau),\lambda_{n}\varphi_{n}\Big)_{L^2(\Omega)}&=\Big( \phi(\cdot,\tau),(-\Delta)^s\varphi_{n}\Big)_{L^2(\Omega)}\nonumber\\
&=\Big((-\Delta)^s \phi(\cdot,\tau),\varphi_{n}\Big)_{L^2(\Omega)}-\int_{\R^{N}\setminus\Omega}\left(\phi\mathcal{N}_{s}\varphi_{n}-\varphi_{n}\mathcal{N}_{s}\phi\right)\;dx\nonumber\\
&=-\int_{\R^{N}\setminus\Omega}g\mathcal{N}_{s}\varphi_{n}dx.
\end{align}
From \eqref{eq-20} and \eqref{eq-21} we can deduce that
\begin{align*} 
w(x,t)=-\phi(x,t)+\sum_{n=1}^{\infty}\left(\int_0^t \Big( g(\cdot,\tau),\mathcal{N}_{s}\varphi_{n} \Big)_{L^2(\Omc)}\frac{1}{\lambda_{n}}B_{n}''(t-\tau)d\tau\right)\varphi_{n}(x).
\end{align*}
We have shown  \eqref{eq-22-1}. Since $\phi,w\in C^\infty([0,T];W^{s,2}(\RR^N))$, then $v\in C^\infty([0,T];W^{s,2}(\RR^N))$.\\

{\bf Step 3}: Using \eqref{eq-17} and calculating, we get that for every $t\in [0,T]$,
\begin{align}\label{Int-Es}
\|(-\Delta)_D^sw(\cdot,t)\|_{L^2(\Omega)}^2=&\left\|\sum_{n=1}^\infty\lambda_n\left(\int_0^t \left( \phi_{\tau\tau}(\cdot,\tau),\varphi_{n}\right)_{L^2(\Omega)}B_{n}(t-\tau)d\tau\right)\varphi_{n}\right\|_{L^2(\Omega)}^2\notag\\
\le &\int_0^t\left\|\sum_{n=1}^\infty\lambda_n\left( \phi_{\tau\tau}(\cdot,\tau),\varphi_{n}\right)_{L^2(\Omega)}B_{n}(t-\tau)\;d\tau\varphi_{n}\right\|_{L^2(\Omega)}^2\notag\\
\le&\int_0^t\sum_{n=1}^\infty\Big|\left(\phi_{\tau\tau}(\cdot,\tau),\varphi_{n}\right)_{L^2(\Omega)}\Big|^2\Big|\lambda_nB_{n}(t-\tau)\Big|^2\;d\tau.
\end{align}

We claim that there is a constant $C>0$ (independent of $n$) such that
\begin{align}\label{eqB}
\left|\lambda_nB_n(t)\right|\le C,\;\;\forall\;n\in\NN\;\mbox{ and }\; t\in [0,T].
\end{align}

We have the following situations.
\begin{itemize}
\item If $\delta=0$, then ${\bf D_n^0}<0$ so that $\alpha_n=0$, $\beta_n=\sqrt{\lambda_n}$, $\lambda_n^{\pm}=0$ for every $n\in\NN$. Thus
\begin{align*}
\left|\lambda_nB_n(t)\right|^2=\frac{\lambda_n^2\sin^2(\sqrt{\lambda_n} t)}{\lambda_n^2}\le 1,\;\;\forall\;n\in\NN.
\end{align*}
\item If $\delta>0$ and $N_0=0$, then since $\lambda_n^{\pm}<0$, we have that
\begin{align*}
\left|\lambda_nB_n(t)\right|^2=\lambda_n^2\frac{e^{2\lambda_n^- t}-2e^{(\lambda_n^-+\lambda_n^+)t}+e^{2\lambda_n^+t}}{(\lambda_n^--\lambda_n^+)^2}\le \frac{4\lambda_n^2}{\delta^2\lambda_n^2-4\lambda_n}\le \frac{4\lambda_1}{\delta\lambda_1-4},\;\;\forall\;n\in\NN.
\end{align*}
\item If $\delta>0$ and $0<N_0<\infty$, then since $\alpha_n<0$, we have that
\begin{align*}
\left|\lambda_nB_n(t)\right|^2=\frac{\lambda_n^2\sin^2(\beta_n t)}{\beta_n^2}e^{2\alpha_n t}\le \frac{\lambda_n^2}{\beta_n^2}=\frac{4\lambda_n}{4-\delta^2\lambda_n}\le\frac{4\lambda_{N_0}}{4-\delta^2\lambda_{N_0}} ,\;\;\forall\;1\le n\le N_0,
\end{align*}
and since $\lambda_n^{\pm}<0$ for all $n>N_0$, we have that
\begin{align*}
\left|\lambda_nB_n(t)\right|^2=\lambda_n^2\frac{e^{2\lambda_n^- t}-2e^{(\lambda_n^-+\lambda_n^+)t}+e^{2\lambda_n^+t}}{(\lambda_n^--\lambda_n^+)^2}\le \frac{4\lambda_n^2}{\delta^2\lambda_n^2-4\lambda_n}\le \frac{4\lambda_1}{\delta\lambda_1-4},\;\;\forall\;n>N_0.
\end{align*}
\end{itemize}
The proof of the claim \eqref{eqB} is complete.

It follows from \eqref{Int-Es}, \eqref{eqB} and \eqref{E-DP-D} that for every $t\in [0,T]$,
\begin{align}\label{Int-Es-2}
\|(-\Delta)_D^sw(\cdot,t)\|_{L^2(\Omega)}^2\le &C\int_0^t\|\phi_{\tau\tau}(\cdot,\tau)\|_{L^2(\Omega)}^2\;d\tau\le C\int_0^t\|\phi_{\tau\tau}(\cdot,\tau)\|_{W^{s,2}(\RR^N)}^2\;d\tau\notag\\
\le & C\int_0^t\|g_{\tau\tau}(\cdot,\tau)\|_{W^{s,2}(\Omc)}^2\;d\tau\le CT\|g_{tt}\|_{L^\infty((0,T);W^{s,2}(\Omc))}^2.
\end{align}
Using \eqref{Int-Es-2}, we get that for every $t\in [0,T]$,
\begin{align*}
\|v(\cdot,t)\|_{W^{s,2}(\RR^N)}\le& C\left(\|(-\Delta)_D^sw(\cdot,t)\|_{L^2(\Omega)}+\|\phi(\cdot,t)\|_{W^{s,2}(\RR^N)}\right)\\
\le &C\left(\|g_{tt}\|_{L^\infty((0,T);W^{s,2}(\Omc))}+\|g(\cdot,t)\|_{W^{s,2}(\Omc)}\right).
\end{align*}
We have shown  \eqref{ESt-Ind} for $m=0$.
Proceeding by induction on $m$ we get  \eqref{ESt-Ind} for $m\in\NN_0$. The proof is finished.
\end{proof}

We conclude this subsection with the following result.

\begin{corollary}
For  every $(u_0, u_1)\in W_0^{s,2}(\bOm)\times L^2(\Omega)$ and $g\in\mathcal D((\Omc)\times(0,T))$, the system \eqref{SD-WE} has a unique solution $(u,u_t)$ given by
 \begin{align*}
u(x,t)=&\sum_{n=1}^{\infty}\Big(A_{n}(t)u_{0,n}+B_{n}(t)u_{1,n}\Big)\varphi_{n}(x)\notag\\
&+\sum_{n=1}^{\infty}\left(\int_0^t \Big( g(\cdot,\tau),\mathcal{N}_{s}\varphi_{n} \Big)_{L^2(\Omc)}\frac{1}{\lambda_{n}}B_{n}''(t-\tau)d\tau\right)\varphi_{n}(x).
\end{align*}
\end{corollary}

\begin{proof}
Since a solution $(u,u_t)$ of \eqref{SD-WE} can de decomposed into $u=v+w$ where $(v,v_t)$ solved \eqref{main-EQ-2} and $(w,w_t)$ is a solution of \eqref{main-EQ-3},  the result follows from  Proposition \ref{pro-sol-SG} and Theorem \ref{theo-44}.
\end{proof}

\subsection{Series solutions of the dual system}

Now we consider the dual system \eqref{ACP-Dual}. That is, the  backward system 
\begin{equation}\label{Dual}
\begin{cases}
\psi_{tt} +(-\Delta)^s\psi-\delta(\Delta)^s\psi_t=0\;\;&\mbox{ in }\; \Omega\times (0,T),\\
\psi=0&\mbox{ in }\;(\Omc)\times (0,T),\\
\psi(\cdot,T)=\psi_0,\;\;\;\; -\psi_t(\cdot,T)=\psi_1\;&\mbox{ in }\;\Omega,
\end{cases}
\end{equation}

Let $\psi_{0,n}:=(\psi_0,\varphi_n)_{L^2(\Omega)}\;\mbox{ and }\; \psi_{1,n}:=(\psi_1,\varphi_n)_{L^2(\Omega)}$.
We have the following existence result.

\begin{theorem}\label{theo-48}
For every $(\psi_0,\psi_1)\in W_0^{s,2}(\bOm)\times L^2(\Omega)$, the dual system \eqref{Dual} has a unique weak solution $(\psi,\psi_t)$ given by
\begin{align}\label{eq-25}
\psi(x,t)=\sum_{n=1}^{\infty}\Big(\psi_{0,n}C_{n}(T-t)+\psi_{1,n}D_{n}(T-t)\Big)\varphi_{n}(x),
\end{align}
where $C_{n}(t)=A_{n}(t)$ and  $D_{n}(t)=-B_{n}(t)$ and we recall that $A_n(t)$ and $B_n(t)$ are given in  \eqref{10} and \eqref{11}, respectively.
In addition the following assertions hold.
\begin{enumerate}
\item There is a constant $C>0$ such that for all $t\in [0,T]$,
\begin{equation}\label{Dual-EST-1}
 \|\psi(\cdot,t)\|_{W_0^{s,2}(\bOm)}^2+ \|\psi_t(\cdot,t)\|_{L^2(\Omega)}^2\le C\left(\|\psi_0\|_{W_0^{s,2}(\bOm)}^2+\|\psi_1\|_{L^2(\Omega)}^2\right),
\end{equation}
and 
\begin{equation}\label{Dual-EST-1-2}
 \|\psi_{tt}(\cdot,t)\|_{W^{-s,2}(\bOm)}^2\le \left(\|\psi_0\|_{W_0^{s,2}(\bOm)}^2+\|\psi_1\|_{L^2(\Omega)}^2\right).
 \end{equation}
\item We have that $\psi\in C([0,T); D((-\Delta)_D^s))\cap L^\infty((0,T);L^2(\Omega))$.
 
\item The mapping 
\[[0,T)\ni t\mapsto\mathcal N_s\psi(\cdot,t)\in L^2(\Omc),\]
 can be analytically extended to the half-plane $\Sigma_T:=\{z\in\CC:\;{Re}(z)<T\}$.
\end{enumerate}
\end{theorem}

The proof of the theorem uses heavily the following result.

\begin{lemma}
There is a constant $C>0$ (independent of $n$) such that for every $t\in [0,T]$, 

\begin{align}\label{ES-C}
\max\left\{\left|C_n(t)\right|^2,\left|\lambda_n^{\frac 12}C_n(t)\right|^2, \left|C_n'(t)\right|^2\right\}\le C,
\end{align}
and 

\begin{align}\label{ES-D}
\max\left\{\left|\lambda_n^{\frac 12}D_n(t)\right|^2,\left|\lambda_nD_n(t)\right|^2, \left|D_n'(t)\right|^2,\left|\lambda_n^{\frac 12}D_n'(t)\right|^2\right\}\le C.
\end{align}
\end{lemma}

\begin{proof}
Firstly, we notice that it suffices to prove  \eqref{ES-C} and \eqref{ES-D} for $n>N_0$. Secondly, we recall that $\lambda_n^{\pm}<0$ for every $n\in\NN$. Thirdly, it is easy to show that there is a constant $C>0$ such that $\Big|\lambda_n^{\pm}e^{\lambda_n^{\pm t}}\Big|\le C$ for every $n>N_0$. From this estimate, we can deduce \eqref{ES-C} and \eqref{ES-D} by using some easy computations as the proof of  \eqref{eqB}. 
\end{proof}

\begin{proof}[\bf Proof of Theorem \ref{theo-48}]
Let

\begin{align*} 
\psi_0=\sum_{n=1}^{\infty}\psi_{0,n}\varphi_{n}, \quad \psi_1=\sum_{n=1}^{\infty}\psi_{1,n}\varphi_{n}.
\end{align*}
We proof the theorem in several steps. Here we include more details.\\

{\bf Step 1}: Proceeding in the same way as the proof of Proposition \ref{pro-sol-SG}, we easily get that 
\begin{align}\label{25}
\psi(x,t)=\sum_{n=1}^{\infty}\Big[C_{n}(T-t)\psi_{0,n}+D_{n}(T-t)\psi_{1,n}\Big]\varphi_{n}(x),
\end{align}
where we recall that $C_{n}(t)=A_{n}(t)$ and  $D_{n}(t)=-B_{n}(t)$. In addition, a simple calculation gives $\psi(x,T)=\psi_0(x)$ and $\psi_t(x,T)=-\psi_1(x)$ for a.e. $x\in\Omega$.

Let us show that $\psi$ satisfies the regularity and variational identity requirements. Let $1\le n\le m$ and set
\begin{align*}
\psi_m(x,t)=\sum_{n=1}^{m}\Big[C_{n}(T-t)\psi_{0,n}+D_{n}(T-t)\psi_{1,n}\Big]\varphi_{n}(x).
\end{align*}
For every $m,\tilde m\in\NN$ with $m>\tilde m$ and $t\in [0,T]$, we have that
\begin{align}\label{Nor1}
\|\psi_m(x,t)-\psi_{\tilde m}(x,t)\|_{W_0^{s,2}(\bOm)}^2=&\sum_{n=\tilde m+1}^m\Big|\lambda_n^{\frac 12}C_{n}(T-t)\psi_{0,n}+\lambda_n^{\frac 12}D_{n}(T-t)\psi_{1,n}\Big|^2\\
\le &2\sum_{n=\tilde m+1}^m\Big|\lambda_n^{\frac 12}C_{n}(T-t)\psi_{0,n}\Big|^2+2\sum_{n=\tilde m+1}^m\Big|\lambda_n^{\frac 12}D_{n}(T-t)\psi_{1,n}\Big|^2.\notag
\end{align}
Using \eqref{ES-C} and \eqref{ES-D} we get from \eqref{Nor1} that for every $m,\tilde m\in\NN$ with $m>\tilde m$ and $t\in [0,T]$,
\begin{align*}
\|\psi_m(x,t)-\psi_{\tilde m}(x,t)\|_{W_0^{s,2}(\bOm)}^2\le C\left(\sum_{n=\tilde m+1}^m\Big|\lambda_n^{\frac 12}\psi_{0,n}\Big|^2+\sum_{n=\tilde m+1}^m\Big|\psi_{1,n}\Big|^2\right)\longrightarrow 0 \;\mbox{ as }\; m,\tilde m\to\infty.
\end{align*}
We have show that the series
\begin{align*}
\sum_{n=1}^{\infty}\Big[C_{n}(T-t)\psi_{0,n}+D_{n}(T-t)\psi_{1,n}\Big]\varphi_{n}\longrightarrow v(\cdot,t)\;\mbox{ in }\; W_0^{s,2}(\bOm),
\end{align*}
 and that the convergence is uniform in $t\in [0,T]$. Hence, $\psi\in C([0,T];W_0^{s,2}(\bOm))$. Using \eqref{ES-C} and \eqref{ES-D} again we get that there is a constant $C>0$ such that for every $t\in [0,T]$,
 \begin{align}\label{A1}
 \|\psi(\cdot,t)\|_{W_0^{s,2}(\bOm)}\le C\Big(\|\psi_0\|_{W_0^{s,2}(\bOm)}+\|\psi_1\|_{L^2(\Omega)}\Big).
 \end{align}
 
{\bf Step 2}:  Next, we claim that $\psi_t\in C([0,T];L^2(\Omega))$. Calculating, we get that 
 \begin{align*}
 (\psi_m)_t(x,t)=-\sum_{n=1}^{m}\Big[C_{n}'(T-t)\psi_{0,n}+D_{n}'(T-t)\psi_{1,n}\Big]\varphi_{n}(x).
 \end{align*}
 Proceeding as above, we can easily deduce that the series
 \begin{align*}
 \sum_{n=1}^{\infty}\Big[C_{n}'(T-t)\psi_{0,n}+D_{n}'(T-t)\psi_{1,n}\Big]\varphi_{n}\longrightarrow \psi_t(\cdot,t)\;\mbox{ in }\;L^2(\Omega),
 \end{align*}
 and the convergence is uniform in $t\in [0,T]$. In addition using \eqref{ES-C} and \eqref{ES-D}, we get that there is a constant $C>0$ such that for every $t\in [0,T]$,
\begin{align}\label{A2}
\|\psi_t(\cdot,t)\|_{L^2(\Omega)}^2\le C\Big(\|\psi_0\|_{W_0^{s,2}(\bOm)}^2+\|\psi_1\|_{L^2(\Omega)}^2\Big).
\end{align}
The estimate \eqref{Dual-EST-1} follows from \eqref{A1} and \eqref{A2}.\\

{\bf Step 3}: We show that $\psi_{tt}\in C([0,T);W^{-s,2}(\bOm))$. Using \eqref{norm-V-2}, \eqref{ES-C} and \eqref{ES-D}, we get that for every $t\in [0,T]$,
\begin{align}\label{MW1}
\|(-\Delta)_D^s\psi(\cdot,t)\|_{W^{-s,2}(\bOm)}^2\le &2\sum_{n=1}^\infty\left(\Big|\lambda_n^{-\frac 12}\lambda_nC_n(T-s)\psi_{0,n}\Big|^2+\Big|\lambda_n^{-\frac 12}\lambda_nD_n(T-s)\psi_{1,n}\Big|^2\right)\notag\\
\le &2\sum_{n=1}^\infty\left(\Big|\lambda_n^{\frac 12}C_n(T-s)\psi_{0,n}\Big|^2+\Big|\lambda_n^{\frac 12}D_n(T-s)\psi_{1,n}\Big|^2\right)\notag\\
\le &C\Big(\|\psi_0\|_{W_0^{s,2}(\bOm)}^2+\|\psi_1\|_{L^2(\Omega)}^2\Big).
\end{align}
Using \eqref{norm-V-2}, \eqref{ES-C} and \eqref{ES-D} again we get that there is a constant $C>0$ such that for every $t\in [0,T]$, 
\begin{align}\label{MW2}
\|(-\Delta)_D^s\psi_t(\cdot,t)\|_{W^{-s,2}(\bOm)}^2\le &2\sum_{n=1}^\infty\left(\Big|\lambda_n^{-\frac 12}\lambda_nC_n'(T-s)\psi_{0,n}\Big|^2+\Big|\lambda_n^{-\frac 12}\lambda_nD_n'(T-s)\psi_{1,n}\Big|^2\right)\notag\\
\le &2\sum_{n=1}^\infty\left(\Big|\lambda_n^{\frac 12}C_n'(T-s)\psi_{0,n}\Big|^2+\Big|\lambda_n^{\frac 12}D_n'  (T-s)\psi_{1,n}\Big|^2\right)\notag\\
\le&C\Big(\|\psi_0\|_{W_0^{s,2}(\bOm)}^2+\|\psi_1\|_{L^2(\Omega)}^2\Big).
\end{align}
Since $\psi_{tt}(\cdot,t)=-(-\Delta)_D^s\psi(\cdot,t)+\delta(-\Delta)_D^s\psi_t(\cdot,t)$, it follows from \eqref{MW1} and \eqref{MW2} that
\begin{align*}
\|\psi_{tt}(\cdot,t)\|_{W^{-s,2}(\bOm)}^2\le C\Big(\|\psi_0\|_{W_0^{s,2}(\bOm)}^2+\|\psi_1\|_{L^2(\Omega)}^2\Big),
\end{align*}
 and we have also shown \eqref{Dual-EST-1-2}.  We can also easily deduce that  $\psi_{tt}\in C([0,T);W^{-s,2}(\bOm))$.\\
 
{\bf Step 4}: We claim that $\psi\in C([0,T); D((-\Delta)_D^s))\cap L^\infty((0,T);L^2(\Omega))$. It follows from the estimate \eqref{Dual-EST-1} that $\psi\in L^\infty((0,T);L^2(\Omega))$. Proceeding as above we get that
\begin{align}\label{M-W}
\|\psi(\cdot,t)\|_{D((-\Delta)_D^s)}^2=&\|(-\Delta)_D^s\psi(\cdot,t)\|_{L^2(\Omega)}^2\notag\\
\le &2\sum_{n=1}^\infty\left(\Big|\lambda_n^{\frac 12}C_n(T-t)\lambda_n^{\frac 12}\psi_{0,n}\Big|^2+\Big|\lambda_nD_n(T-t)\psi_{1,n}\Big|^2\right).
\end{align} 
It follows from \eqref{M-W}, \eqref{ES-C} and \eqref{ES-D} that
\begin{align*}
\|\psi(\cdot,t)\|_{D((-\Delta)_D^s)}^2\le C\Big(\|\psi_0\|_{W_0^{s,2}(\bOm)}^2+\|\psi_1\|_{L^2(\Omega)}^2\Big),
\end{align*}
and we can also deduce that $\psi\in C([0,T); D((-\Delta)_D^s))$. We have shown the claim.\\

{\bf Step 5}: It is easy to see that the mapping $[0,T)\ni t\to \psi(\cdot,t)\in L^2(\Omc)$ can be analytically extended to $\Sigma_T$. We also recall that for every $t\in [0,T)$ fixed, we have that $\psi(\cdot,t)\in D((-\Delta)_D^s)\subset W^{s,2}(\RR^N)$. Therefore, $\mathcal N_sv(\cdot,t)$ exists and belongs to $L^2(\Omc)$.

We claim that 

\begin{align}\label{28}
\mathcal{N}_{s}\psi(x,t)=\sum_{n=1}^{\infty}\Big(C_{n}(T-t)\psi_{0,n}+D_{n}(T-t)\psi_{1,n}\Big)\mathcal{N}_{s}\varphi_{n}(x),
\end{align}
and the series is convergent in $L^2(\R^{N}\setminus\Omega)$ for every $t\in[0,T)$. Indeed, let $\xi>0$ be fixed but arbitrary and let $t\in[0,T-\xi]$. It is sufficient to prove that

\begin{align*} 
\left\|\sum_{n=N_0+1}^{\infty}\Big[C_{n}(T-t)\psi_{0,n}+D_{n}(T-t)\psi_{1,n}\Big]\mathcal{N}_{s}\varphi_{n}\right\|_{L^2(\R^{N}\setminus\Omega)}\longrightarrow 0\text{ as }N_0\to\infty.
\end{align*}
Since $\mathcal{N}_{s}:W^{s,2}(\R^{N})\to L^2(\R^{N}\setminus\Omega)$ is bounded, then using \eqref{ES-C} and \eqref{ES-D}, we get that there is a constant $C>0$ such that
\begin{align*}
&\left\|\sum_{n=N_0+1}^{\infty}\Big(C_{n}(T-t)\psi_{0,n}+D_{n}(T-t)\psi_{1,n}\Big)\mathcal{N}_{s}\varphi_{n}\right\|_{L^2(\R^{N}\setminus\Omega)}^2\\
\le &C\left\|\sum_{n=N_0+1}^{\infty}\Big(C_{n}(T-t)\psi_{0,n}+D_{n}(T-t)\psi_{1,n}\Big)\varphi_{n}\right\|_{W_0^{s,2}(\bOm)}^2\\
\leq &C\left(\sum_{n=N_0+1}^{\infty}|\psi_{0,n}|^2+\sum_{n=N_0+1}^{\infty}|\psi_{1,n}|^2\right)\longrightarrow 0\text{ as }N_0\to \infty.
\end{align*}
Thus, $\mathcal{N}_{s}$ is given by \eqref{28} and the series is convergent in $L^2(\R^{N}\setminus\Omega)$ uniformly on any compact subset of $[0,T)$.

Besides, we obtain the following continuous dependence on the data. Let $m\in\N$ be such that $m>N_0$ and consider
\begin{align}\label{27.1}
\psi_{m}(x,t)=\sum_{k=1}^{m}\Big(C_{n}(T-t)\psi_{0,n}+D_{n}(T-t)\psi_{1,n}\Big)\mathcal{N}_{s}\varphi_{n}(x).
\end{align}
Using the fact that the operator $\mathcal{N}_{s}:  W^{s,2}_0(\bOm)\longrightarrow L^2(\R^{N}\setminus\Omega)$ is bounded,  the continuous embedding $W_0^{s,2}(\bOm)\hookrightarrow L^{2}(\Omega)$, \eqref{ES-C} and \eqref{ES-D}, we get that there is a constant $C>0$ such that for every $t\in [0,T]$,

\begin{align}
&\left\|\sum_{k=1}^{m} C_{k}(T-t)\psi_{0,k}\mathcal{N}_{s}\varphi_{k}\right\|_{L^2(\R^{N}\setminus\Omega)}^2\nonumber\\ 
&\leq 2\left\|\sum_{k=1}^{N_0}C_{k}(T-t)\psi_{0,k}\mathcal{N}_{s}\varphi_{k}\right\|_{L^2(\R^{N}\setminus\Omega)}^2+2\left\|\sum_{k=N_0+1}^{m}C_{k}(T-t)\psi_{0,k}\mathcal{N}_{s}\varphi_{k}\right\|_{L^2(\R^{N}\setminus\Omega)}^2\nonumber\\
&\leq C\left\|\sum_{k=1}^{N_0}C_{k}(T-t)\psi_{0,k}\varphi_{k}\right\|_{W_0^{s,2}(\bOm)}^2+C\left\|\sum_{k=N_0+1}^{m}C_{k}(T-t)\psi_{0,k}\varphi_{k}\right\|_{W_0^{s,2}(\bOm)}^2\nonumber\\
&\leq C \left(\sum_{k=1}^{N_0}|C_{k}(T-t)\lambda_k^{\frac 12}\psi_{0,k}|^2+\sum_{k=N_0+1}^{m}|C_{k}(T-t)\lambda_k^{\frac 12}\psi_{0,k}|^2\right)
\leq C \|\psi_0\|_{W^{s,2}(\bOm)}^2.\label{27.2}
\end{align} 

Analogously, we obtain that there is a constant $C>0$ such that for every $t\in [0,T]$,

\begin{align}\label{27.3}
\left\|\sum_{k=1}^{m} D_{k}(T-t)\psi_{1,k}\mathcal{N}_{s}\varphi_{k}\right\|_{L^2(\R^{N}\setminus\Omega)}^2
\leq C\|\psi_1\|_{L^2(\Omega)}^2.
\end{align}

It follows from \eqref{27.2} and \eqref{27.3} that

\begin{align}\label{27.4}
\|\mathcal{N}_{s}\psi(\cdot,t)\|_{L^2(\R^{N}\setminus\Omega)}^2\leq C\Big(\|\psi_0\|_{W^{s,2}(\bOm)}^2+\|\psi_1\|_{L^2(\Omega)}^2\Big).
\end{align}

Next, since the functions $C_n(z)$ and $D_n(z)$ are entire functions, it follows that the function

\begin{align*}
\sum_{n=1}^{m}\Big[C_{n}(T-z)\psi_{0,n}+D_{n}(T-z)\psi_{1,n}\Big]\mathcal{N}_{s}\varphi_{n}
\end{align*}
is analytic in $\Sigma_T$. 

Let $\tau>0$ be fixed but otherwise arbitrary. Let $z\in\CC$ satisfy $\mbox{Re}(z)\le T-\tau$. Then proceeding as above by using \eqref{ES-C} and \eqref{ES-D}, we get that

 \begin{align*}
& \left\Vert\sum_{n=m+1}^\infty \psi_{0,n}C_n(T-z)\mathcal N_s \varphi_n\right\Vert_{L^2(\Omc)}^2
+\left\Vert\sum_{n=m+1}^\infty \psi_{1,n}D_n(T-z)\mathcal N_s \varphi_n\right\Vert_{L^2(\Omc)}^2\\
 \le &C\sum_{n=m+1}^\infty \left|\lambda_n^{\frac 12}\psi_{0,n}\right|^2 +C\sum_{n=m+1}^\infty |\psi_{1,n}|^2\longrightarrow 0\;\mbox{ as }\; m\to\infty.
 \end{align*}
 We have shown that
 
 \begin{align}\label{form-nor}
\mathcal N_s \psi(\cdot,z)=&\sum_{n=1}^\infty \psi_{0,n}C_n(T-z)\mathcal N_s \varphi_n+\sum_{n=1}^\infty \psi_{1,n}D_n(T-z)\mathcal N_s \varphi_n,
 \end{align}
and the series is uniformly convergent in any compact subset of $\Sigma_T$. Thus, $\mathcal N_s \psi$ given in \eqref{form-nor} is also analytic in $\Sigma_T$.  The proof is finished.
\end{proof}

\section{Proof of the main results}\label{prof-ma-re}

In this section we prove the main results stated in Section \ref{sec-2}. 

\subsection{The lack of exact or null controllability}
We start with the proof of the lack of null/exact controllability of the system \eqref{SD-WE}. For this purpose, we will use the following concept of controllability.

\begin{definition}\label{esp-con}
The system \eqref{SD-WE} is said to be \emph{spectrally controllable} if any finite linear combination of eigenvectors, that is,
\begin{align*}
u_0=\sum_{n=1}^{M}u_{0,n}\varphi_{n}, \quad u_1=\sum_{n=1}^{M}u_{1,n}\varphi_{n},\;\; M\ge 1\;\mbox{ arbitrary},
\end{align*}
can be steered to zero by a control function $g$.
\end{definition}

Let $(u,u_t)$ and $(\psi,\psi_t)$ be the weak solutions of \eqref{SD-WE} and \eqref{ACP-Dual}, respectively. Multiplying the first equation in \eqref{SD-WE} by $\psi$, then integrating by parts over $(0,T)$ and over $\Om$ and using the integration by parts formulas \eqref{Int-Part}-\eqref{Int-Part-2}, we get 

\begin{align}\label{32}
\int_\Omega \Big(-u_{t}\psi+u\psi_{t}-\delta u(-\Delta)^{s}\psi\Big)\Big|_{t=0}^{t=T}dx=\int_0^{T}\int_{\Omc}\Big(g(x,t)+\delta g_t(x,t)\Big)\mathcal{N}_{s}\psi(x,t)dxdt.
\end{align} 

Using the identity \eqref{32} and a density argument to pass to the limit, we obtain the following criterion of null and exact controllabilities.

\begin{lemma}\label{L1}
The following assertions hold.
\begin{enumerate}
\item  The system \eqref{SD-WE} is null controllable if and only if for each initial condition $(u_0,u_1)\in W_0^{s,2}(\bOm)\times L^2(\Omega) $, there exists a control function $g$ such that the solution $(\psi,\psi_t)$ of the dual system \eqref{ACP-Dual} satisfies

\begin{align}\label{33}
(u_1,\psi(\cdot,0))_{L^2(\Om)}&-\langle u_0,\psi_{t}(\cdot,0)\rangle_{\frac 12,-\frac 12}+\langle u_0,\delta(-\Delta)^{s}\psi(\cdot,0)\rangle_{\frac 12,-\frac 12}\notag\\
=&\int_0^{T}\int_{\Omc}\Big(g(x,t)+\delta g_t(x,t)\Big)\mathcal{N}_{s}\psi(x,t)dxdt,
\end{align}
for each $(\psi_0,\psi_1)\in L^2(\Omega)\times W^{-s,2}(\bOm)$.

\item The system \eqref{SD-WE} is exact controllable at time $T>0$, if and only if there exists a control function $g$ such that the solution $(\psi,\psi_t)$ of \eqref{ACP-Dual} satisfies

\begin{align}\label{33-2}
-(u_t(\cdot,T),\psi_0)_{L^2(\Om)}&+\langle u(\cdot,T),\psi_1\rangle_{\frac 12,-\frac 12}-\langle u(\cdot,T),\delta(-\Delta)^{s}\psi_0\rangle_{\frac 12,-\frac 12}\notag\\
=&\int_0^{T}\int_{\Omc}\Big(g(x,t)+\delta g_t(x,t)\Big)\mathcal{N}_{s}\psi(x,t)dxdt,
\end{align}
for each $(\psi_0,\psi_1)\in L^2(\Omega)\times W^{-s,2}(\bOm)$.
\end{enumerate}
\end{lemma}

Now, we are able to give the proof of the first main result of this work.

\begin{proof}[{\bf Proof of Theorem \ref{lact-nul-cont}}]
Firstly, since $\varphi_n\in W_0^{s,2}(\bOm)\subset L^2(\Om)\subset W^{-s,2}(\bOm)$, it suffices to prove that the system is not spectrally controllable.

Secondly, using Definition \ref{esp-con}, we show that no non-trivial finite linear combination of eigenvectors can be driven to zero in finite time. To do this,  we write the solution of \eqref{ACP-Dual} in a better way. With a simple calculation, it is easy to see that

\begin{align}\label{34}
\psi(x,t)=&\sum_{n=1}^{N_0}\left(\widetilde{a}_{n}e^{\widetilde{\lambda}_{n}^+(T-t)}+\widetilde{b}_{n}e^{\widetilde{\lambda}_{n}^-(T-t)}\right)\varphi_{n}(x)\notag\\
&+\sum_{n=N_0+1}^{\infty}\left(a_{n}e^{\lambda_n^+(T-t)}+b_{n}e^{\lambda_n^-(T-t)}\right)\varphi_{n}(x),
\end{align}
where 
\begin{align}\label{35}
\begin{cases}
\displaystyle \widetilde{a}_{n}=\frac{1}{2}\left((1-\frac{\alpha_{n}}{\beta_{n}}i)\psi_{0,n}-\frac{i}{\beta_{n}}\psi_{1,n}\right),\\
\displaystyle\widetilde{b}_{n}=\frac{1}{2}\left((1+\frac{\alpha_n}{\beta_n}i)\psi_{0,n}+\frac{i}{\beta_n}\psi_{1,n}\right),
\end{cases}
\end{align}
and
\begin{align}\label{36}
\begin{cases}
\displaystyle a_{n}=\frac{\lambda_n^-}{\lambda_n^--\lambda_n^+}\psi_{0,n}+\frac{1}{\lambda_n^--\lambda_n^+}\psi_{1,n},\\
\displaystyle b_{n}=\frac{-\lambda_n^+}{\lambda_n^--\lambda_n^+}\psi_{0,n}-\frac{1}{\lambda_n^--\lambda_n^+}\psi_{1,n}.
\end{cases}
\end{align}

Now, we write the initial data in Fourier series
\begin{align}\label{37}
u_0=\sum_{n=1}^{\infty}u_{0,n}\varphi_{n}, \quad u_1=\sum_{n=1}^{\infty}u_{1,n}\varphi_{n},
\end{align}
and suppose that there exists $M\in\N$ such that  
\begin{align}\label{37.1}
u_{0,n}=u_{1,n}=0, \quad \forall\;  n\geq M.
\end{align}

Assume that \eqref{SD-WE} is spectrally controllable. Then, there exists a control function $g$ such that the solution $(u,u_t)$ of \eqref{SD-WE} with $u_0,u_1$ given by \eqref{37}--\eqref{37.1} satisfies $u(\cdot,T)=u_{t}(\cdot,T)=0$ in $\Omega$. From Lemma \ref{L1} we have
\begin{align}\label{38}
(u_1,\psi(\cdot,0))_{L^2(\Om)}&-\langle u_0,\psi_{t}(\cdot,0)\rangle_{\frac 12,-\frac 12}+\langle u_0,\delta(-\Delta)^{s}\psi(\cdot,0)\rangle_{\frac 12,-\frac 12}\notag\\
=&\int_0^{T}\int_{\Omc}\Big(g(x,t)+\delta g_t(x,t)\Big)\mathcal{N}_{s}\psi(x,t)dxdt,
\end{align}
for any solution $(\psi,\psi_t)$ of the dual system \eqref{ACP-Dual}.

We divide the proof in the following two cases.
\begin{enumerate}
\item[{\bf Case 1.}] $M>N_0$. 

We consider the following trajectories: 
\begin{align}
\psi(x,t)=e^{\widetilde{\lambda}_{n}^+(T-t)}\varphi_{n}(x),\quad \psi(x,t)=e^{\widetilde{\lambda}_{n}^-(T-t)}\varphi_{n}(x)\text{ for } n\leq N_0\label{38.1}\\
\psi(x,t)=e^{\lambda_n^+(T-t)}\varphi_{n}(x),\quad\psi(x,t)=e^{\lambda_n^-(T-t)}\varphi_{n}(x)\text{ for }N_0<n<M\label{38.2}.
\end{align}
Replacing \eqref{38.1} in \eqref{38} we obtain, for any $n\in[0,N_0]$, the following system:
\begin{align}
u_{1,n}+ u_{0,n}\widetilde{\lambda}_{n}^++ \delta u_{0,n}\lambda_{n}=\int_0^{T}\int_{\Omc}(g(x,t)+\delta g_t(x,t))e^{-\widetilde{\lambda}_{n}^+ t}\mathcal{N}_{s}\varphi_{n}(x)dxdt,\label{39}\\
u_{1,n}+ u_{0,n}\widetilde{\lambda}_{n}^- +\delta u_{0,n}\lambda_{n}=\int_0^{T}\int_{\Omc}(g(x,t)+\delta g_t(x,t))e^{-\widetilde{\lambda}_{n}^- t}\mathcal{N}_{s}\varphi_{n}(x)dxdt,\label{40}
\end{align}
and replacing \eqref{38.2} in \eqref{38}, it follows that for any $N_0<n<M$,
\begin{align}
u_{1,n}+ u_{0,n}\lambda_n^+ + \delta u_{0,n}\lambda_{n}=\int_0^{T}\int_{\Omc}(g(x,t)+\delta g_t(x,t))e^{-\lambda_n^+ t}\mathcal{N}_{s}\varphi_{n}(x)dxdt,\label{41}\\
u_{1,n}+ u_{0,n}\lambda_n^- + \delta u_{0,n}\lambda_{n}=\int_0^{T}\int_{\Omc}(g(x,t)+\delta g_t(x,t))e^{-\lambda_n^- t}\mathcal{N}_{s}\varphi_{n}(x)dxdt.\label{42}
\end{align}

Next, define the complex function
\begin{align}\label{43}
F(z)=\int_0^{T}\left(\int_{\Omc}(g(x,t)+\delta g_t(x,t))\mathcal{N}_{s}\varphi_{n}(x)dx\right)e^{izt}dt.
\end{align}
According to Paley--Wiener theorem, $F$ is an entire function. Due to \eqref{37.1}, from \eqref{41} and \eqref{42} we obtain that $F$ satisfies $F(i\lambda_n^+)=F(i\lambda_n^-)=0$, for all $n>M$. Besides, we know that $\lambda_n^+\to -\delta$ as $n\to\infty$ (see Remark \ref{rem-41}). Then, $F$ is zero in a set with finite accumulation point. This implies that $F\equiv 0$. In particular, $F(i\widetilde{\lambda}_{n}^+)=F(i\widetilde{\lambda}_{n}^-)=0$. From \eqref{39} to \eqref{42}, it is easy to see that $u_{0,n}=u_{1,n}=0$ for $n\leq N_0$ and $u_{0,n}=u_{1,n}=0$ for $N_0<n<M$. Thus the trivial state is the only one which can be steered to zero.

\item[{\bf Case 2.}] $M=N_0$ or $M<N_0$.

In these cases the identities \eqref{41} and \eqref{42} do not appear. Proceeding as above we get the desired result.
\end{enumerate}
We have shown that the system \eqref{SD-WE} is not spectrally controllable. It is clear from the proof that \eqref{SD-WE} is neither exact nor null controllable. The proof is finished.
\end{proof}

\begin{remark}
{\em 
We can observe that in the case $\delta=0$, the previous conclusion is not valid. This is due to the fact that,  when $\delta=0$, the previous computation gives the following system for $n\le M$:
\begin{align*}
u_{1,n}+ u_{0,n}\widetilde{\lambda}_{n}^+ =\int_0^{T}\int_{\Omc}g(x,t)e^{-\widetilde{\lambda}_{n}^+ t}\mathcal{N}_{s}\varphi_{n}(x)dxdt,\\
u_{1,n}+ u_{0,n}\widetilde{\lambda}_{n}^- =\int_0^{T}\int_{\Omc}g(x,t)e^{-\widetilde{\lambda}_{n}^- t}\mathcal{N}_{s}\varphi_{n}(x)dxdt.
\end{align*}
Since $\widetilde{\lambda}_{n}^{\pm}\to \pm i\infty$, as $n\to\infty$, it follows that the set $\{\widetilde{\lambda}_{n}^{\pm}\}_{n\in\N}$, on which
the function $F$ defined in \eqref{43} is zero, does not have a finite accumulation point. Thus we cannot conclude that $u_{0,n}=u_{1,n}=0$ for all $n\le M$. This shows that the analysis of the null/exact controllability of the pure wave equation (without damping) with the fractional Laplacian must be done by using other techniques as in the classical case $s=1$.}
\end{remark}

We conclude this subsection with the following observation.

\begin{remark}
{\em We mention the following situations.
\begin{enumerate}
\item Firstly, we notice that if $s$ is close to $1$, using the results obtained in \cite{Umb3}, we can deduce that the eigenfuctions $\varphi_n\in W_0^{1,2}(\Om)$, for every $n\in\NN$, and the net $\{\varphi_n\}=\{\varphi_{n,s}\}_{0<s<1}$ converges as $s\uparrow 1^-$ to the eigenfunctions of the Laplace operator with the zero Dirichlet boundary condition.
This implies that if $(\psi_0,\psi_1)\in W_0^{1,2}(\Om)\times L^2(\Om)\hookrightarrow W_0^{s,2}(\bOm)\times L^2(\Omega)$ and $s$ is close to $1$, then the solution $(\psi,\psi_t)$ of the dual system has the following regularity: $\psi\in C([0,T]; W_0^{1,2}(\Omega))\cap C^1([0,T];L^2(\Omega)$ and $\psi_{tt}\in C([0,T);(W_0^{1,2}(\Om))^\star))$. Therefore, if the control function $g$ has enough regularity as in Lemma \ref{lem-37} and $(u_0,u_1)\in W_0^{1,2}(\Om)\times L^2(\Om)$, then in Lemma \ref{L1}, using \eqref{lim1} and \eqref{lim2}, and taking the limit of \eqref{33} and \eqref{33-2} as $ s \uparrow 1^-$, we can deduce that

\begin{align*}
(u_1,\psi(\cdot,0))_{L^2(\Om)}&-\langle u_0,\psi_{t}(\cdot,0)\rangle_{1,-1}-\langle u_0,\delta\Delta\psi(\cdot,0)\rangle_{1,-1}\notag\\
=&\int_0^{T}\int_{\pOm}\Big(g(x,t)+\delta g_t(x,t)\Big)\frac{\partial\psi(x,t)}{\partial\nu}\;d\sigma dt,
\end{align*}
for every $(\psi_0,\psi_1)\in L^2(\Omega)\times (W_0^{1,2}(\Om))^\star$,
and
\begin{align*}
-(u_t(\cdot,T),\psi_0)_{L^2(\Om)}&+\langle u(\cdot,T),\psi_1\rangle_{1,-1}-\langle u(\cdot,T),\delta\Delta\psi_0\rangle_{1,-1}\notag\\
=&\int_0^{T}\int_{\pOm}\Big(g(x,t)+\delta g_t(x,t)\Big)\frac{\partial\psi(x,t)}{\partial\nu}\;d\sigma dt,
\end{align*}
respectively. These are the notions of null and exact controllabilities, respectively, of the following (possible) strong damping local wave equation:

\begin{equation}\label{S-Del}
\begin{cases}
u_{tt} -\Delta u -\delta\Delta u_t=0\;\;&\mbox{ in }\;\Omega\times (0,T),\\
u=g\chi_{\omega\times (0,T)}&\mbox{ on }\;\pOm\times (0,T);\\
u(\cdot,0)=u_0,\;u_t(\cdot,0)=u_1&\mbox{ in }\;\Omega,
\end{cases}
\end{equation}
studied by several authors (see e.g. \cite{rosier2007,Zua1} and the references therein). Here, $\langle\cdot,\cdot\rangle_{1,-1}$ denotes the duality pair between $W_0^{1,2}(\Om)$ and $(W_0^{1,2}(\Om))^\star$.

\item In the above sense, the results obtained in the present paper for the fractional case $0<s<1$ are consistent with the ones obtained for the case of the Laplace operator in one space dimension $N=1$ in \cite{rosier2007}. For this reason, following the techniques we developed in the present article, we anticipate that the approximate controllability or the lack of exact/null controllability of the system \eqref{S-Del} proved in \cite{rosier2007} for one space dimension, that is, $N=1$, is also valid for any dimension $N\ge 1$.
\end{enumerate}
}
\end{remark}

\subsection{The unique continuation property}

Here we show that the dual system satisfies the unique continuation property.

\begin{proof}[{\bf Proof of Theorem \ref{pro-uni-con}}]
Let $\mathcal O\subset\Omc$ be an arbitrary non-empty open set. Suppose that $\mathcal{N}_{s}\psi=0$ in $\mathcal{O}\times(0,T)$. Then,
\begin{align}\label{44}
\mathcal{N}_{s}\psi(x,t)=\sum_{n=1}^{\infty}\Big(C_{n}(T-t)\psi_{0,n}+D_{n}(T-t)\psi_{1,n}\Big)\mathcal{N}_{s}\varphi_{n}(x)=0,\quad \forall\; (x,t)\in\mathcal{O}\times(0,T).
\end{align}

Since $\mathcal{N}_{s}\psi$ can be analytically extended to $\Sigma_{T}$ (by Theorem \ref{theo-48}(c)), it follows from \eqref{44} that
\begin{align}\label{45}
\mathcal{N}_{s}\psi(x,t)=\sum_{n=1}^{\infty}\Big(C_{n}(T-t)\psi_{0,n}+D_{n}(T-t)\psi_{1,n}\Big)\mathcal{N}_{s}\varphi_{n}(x)=0,\quad  (x,t)\in\mathcal{O}\times(-\infty,T).
\end{align}

Let $\{\lambda_{k}\}_{k\in\N}$ be the set of all eigenvalues of the operator $(-\Delta)_D^{s}$ and let $\{\varphi_{k_{j}}\}_{1\leq j\leq m_{k}}$ be an orthonormal basis for ker$(\lambda_{k}-(-\Delta)_D^{s})$. Then, \eqref{45} can be rewritten as 

\begin{align}\label{46}
\mathcal{N}_{s}\psi(x,t)=&\sum_{k=1}^{\infty}\left(\sum_{j=1}^{m_{k}}\psi_{0,k_{j}}\mathcal{N}_{s}\varphi_{k_{j}}(x)\right)C_{k}(T-t)\notag\\
&+\sum_{k=1}^{\infty}\left(\sum_{j=1}^{m_{k}}\psi_{1,k_{j}}\mathcal{N}_{s}\varphi_{k_{j}}(x)\right)D_{k}(T-t)=0,\quad \forall\; (x,t)\in\mathcal{O}\times(-\infty,T).
\end{align}

Let $z\in\C$ with Re$(z)=\eta>0$ and let $m\in\N$. Since $\varphi_{k_{j}}$, $1\leq j\leq m_k$, are orthonormal, then using the fact that the operator $\mathcal{N}_{s}:D((-\Delta)_D^{s})\subset W^{s,2}(\R^{N})\to L^2(\R^{N}\setminus\Omega)$ is bounded, the continuous dependence on the data of $\mathcal{N}_{s}$ (see \eqref{27.4}), and letting 

\begin{align*}
\psi_{m}(\cdot,t):=&\sum_{k=1}^{m}\left(\sum_{j=1}^{m_{k}}\psi_{0,k_{j}}\mathcal{N}_{s}\varphi_{k_{j}}(x)\right)e^{z(t-T)}C_{k}(T-t)\notag\\
&+\sum_{k=1}^{m}\left(\sum_{j=1}^{m_{k}}\psi_{1,k_{j}}\mathcal{N}_{s}\varphi_{k_{j}}(x)\right)e^{z(t-T)}D_{k}(T-t),
\end{align*}
we obtain that there is a constant $C>0$ such that for every $t\in [0,T]$,

\begin{align} \label{48}
\|\psi_{m}(\cdot,t)\|_{L^2(\R^{N}\setminus\Omega)}\leq C e^{\eta(t-T)}\Big(\|\psi_0\|_{W^{s,2}(\bOm)}+\|\psi_1\|_{L^2(\Omega)}\Big).
\end{align}
The right hand side of \eqref{48} is integrable over $t\in(-\infty,T)$ and
\begin{align*}
\int_{-\infty}^{T}e^{\eta(t-T)}\Big(\|\psi_0\|_{W^{s,2}(\bOm)}+\|\psi_1\|_{L^2(\Omega)}\Big)\;dt=\frac{1}{\eta}\left(\|\psi_0\|_{W^{s,2}(\bOm)}+\|\psi_1\|_{L^2(\Omega)}\right).
\end{align*}

By the Lebesgue dominated convergence theorem, we can deduce that
\begin{align*}
\int_{-\infty}^{T}e^{z(t-T)}&\left[\sum_{k=1}^{\infty}\left(\sum_{j=1}^{m_{k}}\psi_{0,k_{j}}\mathcal{N}_{s}\varphi_{k_{j}}(x)\right)C_{k}(T-t)+\sum_{k=1}^{\infty}\left(\sum_{j=1}^{m_{k}}\psi_{1,k_{j}}\mathcal{N}_{s}\varphi_{k_{j}}(x)\right)D_{k}(T-t)\right]dt\\
&=\sum_{k=1}^{\infty}\sum_{j=1}^{m_{k}}\Big(E_{k}(z)\psi_{0,k_{j}}+F_{k}(z)\psi_{1,k_{k}}\Big)\mathcal{N}_{s}\varphi_{k_{j}}, \quad x\in\R^{N}\setminus\Omega, \ \text{Re}(z)>0,
\end{align*}
where 
\begin{align*}
E_{k}(z)=
\begin{cases}
\displaystyle \frac{1-\frac{\alpha_{k}+\beta_{k}}{2\beta_{k}}}{z-(\alpha_{k}+\beta_{k})}+\frac{\frac{\alpha_{k}+\beta_{k}}{2\beta_k}}{z-(\alpha_{k}-\beta_{k})}& \text{if }k\leq N_0\\
\displaystyle  \frac{1}{\lambda_k^--\lambda_k^+}\left(\frac{\lambda_k^-}{z-\lambda_k^+}-\frac{\lambda_k^+}{z-\lambda_k^-}\right)& \text{if }k>N_0,
\end{cases}
\end{align*}
and
\begin{align*}
F_{k}(z)=
\begin{cases}
\displaystyle\frac{-1}{2\beta_{k}(z-(\alpha_{k}+\beta_{k}))}+\frac{1}{2\beta_{k}(z-(\alpha_{k}-\beta_{k}))}& \text{if }k\leq N_0\\
\displaystyle \frac{1}{\lambda_k^--\lambda_k^+}\left(\frac{1}{z-\lambda_k^+}-\frac{1}{z-\lambda_k^-}\right)& \text{if }k>N_0.
\end{cases}
\end{align*}

From \eqref{46} we get that
\begin{align}\label{50}
\sum_{k=1}^{\infty}\sum_{j=1}^{m_{k}}\Big(E_{k}(z)\psi_{0,k_{j}}+F_{k}(z)\psi_{1,k_{k}}\Big)\mathcal{N}_{s}\varphi_{k_{j}}(x)=0,\quad x\in\mathcal{O}, \ \text{Re}(z)>0.
\end{align}
Using the analytic continuation in $z$, we get that \eqref{50} holds for every $z\in\C\setminus\{\lambda_k^+,\lambda_k^-\}_{k>N_0}$ and also for every $z\in\C\setminus\{\alpha_{k}+\beta_{k}, \alpha_{k}-\beta_{k}\}_{k\le N_0}$.  

For $k\le N_0$, we take a small circle about $\alpha_{k}+\beta_{k}$, but not including $\{\alpha_{l}+\beta_{l}\}_{l\neq k}$. Then, integrating over that circle we get
\begin{align}\label{51}
\sum_{j=1}^{m_{k}}\left[\left(1-\frac{\alpha_{k_j}+\beta_{k_j}}{2\beta_{k_j}}\right)\psi_{0,k_{j}}-\frac{1}{2\beta_{k_j}}\psi_{1,k_j}\right]\mathcal{N}_{s}\varphi_{k_{j}}=0,\quad x\in\mathcal{O}.
\end{align}

Now, integrating over a small circle about $\alpha_{k}-\beta_{k}$ and not including $\{\alpha_{l}-\beta_{l}\}_{l \neq k}$, we obtain
\begin{align}\label{52}
\sum_{j=1}^{m_{k}}\left(\frac{\alpha_{k_j}+\beta_{k_j}}{2\beta_{k_j}}\psi_{0,k_j}+\frac{1}{2\beta_{k_j}}\psi_{1,k_j}\right)\mathcal{N}_{s}\varphi_{k_j}=0,\quad x\in\mathcal{O}.
\end{align}

Let
\begin{align*}
\psi_{k}^1&:=\sum_{j=1}^{m_{k}}\left[\left(1-\frac{\alpha_{k_j}+\beta_{k_j}}{2\beta_{k_j}}\right)\psi_{0,k_{j}}-\frac{1}{2\beta_{k_j}}\psi_{1,k_j}\right]\varphi_{k_{j}},\\ 
\psi_k^2&:=\sum_{j=1}^{m_{k}}\left(\frac{\alpha_{k_j}+\beta_{k_j}}{2\beta_{k_j}}\psi_{0,k_j}+\frac{1}{2\beta_{k_j}}\psi_{1,k_j}\right)\varphi_{k_j}.
\end{align*}
It follows from \eqref{51} and \eqref{52} that $\mathcal{N}_{s}\psi_k^1=\mathcal{N}_{s}\psi_k^2=0$ in $\mathcal{O}$. We have shown that 
\begin{align*}
(-\Delta)^{s}\psi_{k}^l=\lambda_{k}\psi_{k}^l \quad\text{  in }\Omega\quad\text{ and }\quad\mathcal{N}_{s}\psi_{k}^l=0\quad\text{ in }\mathcal{O}, \ l=1,2.
\end{align*}
It follows from Lemma \ref{lem-UCD} that $\psi_k^l=0$, for every $k$, $l=1,2$. Using the fact that $\{\varphi_{k_j}\}_{1\leq j\leq m_k}$ is linearly independent in $L^2(\Omega)$, we get that 
\begin{align*}
\left(\left(1-\frac{\alpha_{k_j}+\beta_{k_j}}{2\beta_{k_j}}\right)\psi_{0,k_j}-\frac{1}{2\beta_{k_j}}\psi_{1,k_j},\varphi_{k_j}\right)=0,\quad 1\leq j\leq m_{k},\\
\left(\frac{\alpha_{k_j}+\beta_{k_j}}{2\beta_{k_j}}\psi_{0,k_j}+\frac{1}{2\beta_{k_j}}\psi_{1,k_j},\varphi_{k_j}\right)=0,\quad 1\leq j\leq m_{k}.
\end{align*}
Therefore, we can deduce that 
\begin{align}\label{53}
\psi_{0,k}=\psi_{1,k}=0, \  k\le N_0.
\end{align}

On the other hand, since the real number $\lambda_k^+$ and $\lambda_k^-$ (see \eqref{convergence}) satisfy
\begin{align*}
\lambda_k^+\sim -\delta \quad\text{ and }\quad \lambda_k^-\sim -\lambda_{k} \quad\text{ as }\quad k\to\infty,
\end{align*}
for $k>N_0$, we can take a suitable small circle about $\lambda_k^+$ and not including $\{\lambda_l^+\}_{l\neq k}$ and also not including $\{\lambda_k^-\}_{k>N_0}$, and integrating \eqref{50} over that circle, we get that
\begin{align}\label{54}
\sum_{j=1}^{m_{k}}\left(\frac{\lambda_{k_j}^-}{\lambda_{k_j}^--\lambda_{k_j}^+}\psi_{0,k_{j}}+\frac{1}{\lambda_{k_j}^--\lambda_{k_j}^+}\psi_{1,k_{j}}\right)\mathcal{N}_{s}\varphi_{k_{j}}=0,\quad x\in\mathcal{O}.
\end{align}

Let us consider 
\begin{align*}
\psi_{k}^{1}:=\sum_{j=1}^{m_{k}}\left(\frac{\lambda_{k_j}^-}{\lambda_{k_j}^--\lambda_{k_j}^+}\psi_{0,k_{j}}+\frac{1}{\lambda_{k_j}^--\lambda_{k_j}^+}\psi_{1,k_{j}}\right)\varphi_{k_{j}}.
\end{align*}
I follows from \eqref{54} that $\mathcal{N}_{s}\psi_{k}^1=0$ in $\mathcal{O}$. Thus, $\psi_{k}^1$ solves the elliptic problem
\begin{align*}
(-\Delta)^{s}\psi_{k}^1=\lambda_{k}\psi_{k}^1 \text{ in }\Omega\text{ and }\mathcal{N}_{s}\psi_{k}^1=0\text{ in }\mathcal{O}.
\end{align*}
From Lemma \ref{lem-UCD} we get that $\psi_{k}^1=0$ in $\Omega$ for every $k$. Since $\{\varphi_{k_{j}}\}_{1\leq j\leq m_{k}}$ is linearly independent in $L^2(\Omega)$, we deduce that
\begin{align*}
\left(\frac{\lambda_{k_j}^-}{\lambda_{k_j}^--\lambda_{k_j}^+}\psi_{0,k_j}+\frac{1}{\lambda_{k_j}^--\lambda_{k_j}^+}\psi_{1,k_j},\varphi_{k_{j}}\right)=0, \quad \forall \;1\leq j\leq m_{k}.
\end{align*}
Thus,
\begin{align}\label{55}
\frac{\lambda_k^-}{\lambda_k^--\lambda_k^+}\psi_{0,k}+\frac{1}{\lambda_k^--\lambda_k^+}\psi_{1,k}=0.
\end{align}

Similarly, taking a circle about $\lambda_k^-$ and not including $\{\lambda_l^-\}_{l\neq k}$ and also not including $\{\lambda_k^+\}_{k>N_0}$, we obtain that
\begin{align*}
\sum_{j=1}^{m_{k}}\left(\frac{-\lambda_{k_j}^+}{\lambda_{k_j}^--\lambda_{k_j}^+}\psi_{0,k_{j}}-\frac{1}{\lambda_{k_j}^--\lambda_{k_j}^+}\psi_{1,k_{j}}\right)\mathcal{N}_{s}\varphi_{k_{j}}=0,\quad x\in\mathcal{O}.
\end{align*}
Let
 \begin{align*}
\psi_{k}^{2}:=\sum_{j=1}^{m_{k}}\left(\frac{-\lambda_{k_j}^+}{\lambda_{k_j}^--\lambda_{k_j}^+}\psi_{0,k_{j}}-\frac{1}{\lambda_{k_j}^--\lambda_{k_j}^+}\psi_{1,k_{j}}\right)\varphi_{k_{j}}.
\end{align*}
It follows from \eqref{54} that $\mathcal{N}_{s}\psi_{k}^2=0$ in $\mathcal{O}$. Hence, $\psi_{k}^2$ solves the elliptic problem
\begin{align*}
(-\Delta)^{s}\psi_{k}^2=\lambda_{k}\psi_{k}^2 \text{ in }\Omega\text{ and }\mathcal{N}_{s}\psi_{k}^2=0\text{ in }\mathcal{O}.
\end{align*}
From Lemma \ref{lem-UCD}, we get that $\psi_{k}^2=0$ in $\Omega$ for every $k$. Therefore,
\begin{align}\label{56}
\frac{-\lambda_k^+}{\lambda_k^--\lambda_k^+}\psi_{0,k}-\frac{1}{\lambda_k^--\lambda_k^+}\psi_{1,k}=0.
\end{align}
Finally, from \eqref{55} and \eqref{56} we get that 
\begin{align}\label{57}
\psi_{0,k}=\psi_{1,k}=0, \ k>N_0.
\end{align}
From \eqref{53} and \eqref{57}, we finally obtain that $\psi_0=\psi_1=0$. Since the solution $(\psi,\psi_t)$ of the adjoint system is unique, we can conclude that $\psi=0$ in $\Omega\times(0,T)$. The proof is finished.
\end{proof}

\subsection{The approximate controllability}

We obtain the result as a direct consequence of the unique continuation property for the adjoint system (Theorem \ref{pro-uni-con}).

\begin{proof}[{\bf Proof of Theorem \ref{main-Theo}}]
Let $g\in \mathcal D(\mathcal O\times (0,T))$,  $(u,u_t)$ the unique weak solution of \eqref{SD-WE} with $u_0=u_1=0$ and let $(\psi,\psi_t)$ be the unique weak solution of \eqref{ACP-Dual} with $(\psi_0,\psi_1)\in W_0^{s,2}(\bOm)\times  L^2(\Omega)$. Firstly, it follows from Theorems \ref{theo-44}  that $u\in C^\infty([0,T];W^{s,2}(\RR^N))$. Thus $u(\cdot,T)\in L^2(\Omega)$ and $u_t(\cdot,T)\in W^{-s,2}(\bOm)$.
Secondly, it follows from Theorem \ref{theo-48} that $\psi\in L^\infty((0,T);L^2(\Omega))$.   
Therefore, using the identity  \eqref{32} we can deduce  that

\begin{align}\label{58}
-\langle u_t(\cdot,T),\psi_0\rangle_{-\frac 12,\frac 12}&+(u(\cdot,T),\psi_1)_{L^2(\Om)}-\langle u(\cdot,T),\delta(-\Delta)^{s}\psi_0\rangle_{\frac 12,-\frac 12}\notag\\
=&\int_0^{T}\int_{\Omc}\Big(g(x,t)+\delta g_t(x,t)\Big)\mathcal{N}_{s}\psi(x,t)dxdt.
\end{align}
If $(\psi_0,\psi_1)\in D((-\Delta)_D^s)\times L^2(\Om)\hookrightarrow W_0^{s,2}(\bOm)\times  L^2(\Omega)$, then \eqref{58} becomes
\begin{align}\label{58-2}
-\langle u_t(\cdot,T),\psi_0\rangle_{-\frac 12,\frac 12}&+\Big(u(\cdot,T),\psi_1-\delta(-\Delta)^s\psi_0\Big)_{L^2(\Om)}\notag\\
=&\int_0^{T}\int_{\Omc}\Big(g(x,t)+\delta g_t(x,t)\Big)\mathcal{N}_{s}\psi(x,t)dxdt.
\end{align}
Since $D((-\Delta)_D^s)\times L^2(\Om)$ is dense $W_0^{s,2}(\bOm)\times  L^2(\Omega)$, to prove that the set $\Big\{(u(\cdot,T),u_t(\cdot,T)):\; g\in \mathcal D(\mathcal O\times (0,T))\Big\}$ is dense in $L^2(\Omega)\times W^{-s,2}(\bOm)$, it suffices to show that if $(\psi_0,\psi_1)\in D((-\Delta)_D^s)\times L^2(\Om)$ is such that
\begin{align}\label{eq42}
-\langle u_t(\cdot,T),\psi_0\rangle_{-\frac 12,\frac 12}+\Big(u(\cdot,T),\psi_1-\delta(-\Delta)^s\psi_0\Big)_{L^2(\Om)}=0,
\end{align}
for any $g\in \mathcal D(\mathcal O\times (0,T))$, then $\psi_0=\psi_1=0$. 

Indeed, let $(\psi_0,\psi_1)\in D((-\Delta)_D^s)\times L^2(\Om)$ satisfy \eqref{eq42}. It follows from \eqref{58-2} and \eqref{eq42} that
\begin{align*}
\int_0^{T}\int_{\Omc}\Big(g(x,t)+\delta g_t(x,t)\Big)\mathcal{N}_{s}\psi(x,t)dxdt=0,
\end{align*}
for any $g\in \mathcal D(\mathcal O\times (0,T))$. By the fundamental lemma of the calculus of variations, we have that
\begin{align*}
\mathcal N_s\psi=0\;\;\mbox{ in }\; \mathcal O\times (0,T).
\end{align*}
It follows from Lemma \ref{lem-UCD} that $\psi=0$ in $\mathcal{O}\times(0,T)$. Since the solution $(\psi,\psi_t)$ of \eqref{ACP-Dual} is unique, we can conclude that $\psi_0=\psi_1=0$ in $\Omega$. The proof is finished.
\end{proof}

\bibliographystyle{plain}
\bibliography{biblio}

\end{document}